\documentclass[11pt, reqno, english]{amsart}  
\usepackage[utf8]{inputenc}
\usepackage[T1]{fontenc}
\usepackage{amsmath,amsthm}
\usepackage{amsfonts,amssymb}
\usepackage{url}
\usepackage{mathtools}  
\usepackage[colorlinks=true,urlcolor=blue,linkcolor=red,citecolor=magenta]{hyperref}
\usepackage{enumerate, paralist}

\usepackage[left=1in,right=1in,top=1in,bottom=1in]{geometry}
\numberwithin{equation}{section}

\theoremstyle{plain}
\newtheorem{theorem}{Theorem}[section]
\newtheorem{lemma}[theorem]{Lemma}
\newtheorem{corollary}[theorem]{Corollary}
\newtheorem{proposition}[theorem]{Proposition}

\theoremstyle{definition}

\newtheorem{remark}[theorem]{Remark}

\newcommand{\R}{\mathbb{R}}
\newcommand{\C}{\mathbb{C}}
\newcommand{\Z}{\mathbb{Z}}
\newcommand{\K}{\mathbb{K}}
\newcommand{\F}{\mathcal{F}}
\newcommand{\KG}{\mathrm{KG}}
\newcommand{\supp}{\mathrm{supp}}
\newcommand{\Avg}{\mathrm{Avg}}

\DeclareMathOperator{\codim}{\mathrm{codim}}
\DeclareMathOperator{\Int}{\mathrm{Int}}
\DeclareMathOperator{\Row}{\mathrm{Row}}
\DeclareMathOperator{\Ver}{\mathrm{Vert}}

\DeclareMathOperator{\Aff}{\mathrm{Aff}}
\DeclareMathOperator{\Conf}{\mathrm{Conf}}

\title{Fan distributions via Tverberg partitions and Gale duality}

\author[Huang]{Shuai Huang}
\address[EH]{Dept.\ Math.\, Bard College, Annandale-on-Hudson, NY 12504, USA}
\email{eh6041@bard.edu} 

\author[Miller]{Jasper Miller}
\address[JM]{Dept.\ Math.\, Bard College, Annandale-on-Hudson, NY 12504, USA}
\email{jm5305@bard.edu} 

\author[Rose-Levine]{Daniel Rose-Levine}
\address[DR]{Dept.\ Math.\, Bard College, Annandale-on-Hudson, NY 12504, USA}
\email{dr6048@bard.edu} 

\author[Simon]{Steven Simon}
\address[SS]{Dept.\ Math.\, Bard College, Annandale-on-Hudson, NY 12504, USA}
\email{ssimon@bard.edu}

\begin{document}

\begin{abstract}

Equipartition theory, beginning with the classical ham sandwich theorem, seeks the fair division of finite point sets in $\R^d$ by the full-dimensional regions determined by a prescribed geometric dissection of $\R^d$. Here we examine \emph{equidistributions} of finite point sets in $\R^d$ by prescribed \emph{low dimensional} subsets of $\R^d$. Our main result states that if $r\geq 3$ is a prime power, then for any $m$-coloring of a sufficiently small point set $X$ in $\R^d$, there exists an $r$-fan in $\R^d$ -- that is, the union of $r$ ``half-flats'' of codimension $r-2$ centered about a common $(r-1)$-codimensional affine subspace -- which captures all the points of $X$ in such a way that each half-flat contains at most an $r$-th of the points from each color class. The number of points in $\R^d$ we require for this is essentially sharp when $m\geq 2$ and in particular is asymptotically tight.  Additionally, we extend our equidistribution results to ``piercing''  distributions in a similar fashion to Dolnikov's hyperplane transversal generalization of the ham sandwich theorem. By analogy with recent work of Frick et al., our results are obtained by applying Gale duality to linear cases of topological Tverberg-type theorems. Finally, we extend our distribution results to multiple $r$-fans after establishing a multiple intersection version of a topological Tverberg-type theorem due to Sarkaria. 
\end{abstract}

\maketitle

\section{Introduction and Statement of Results}
\label{sec:intro}

\subsection{From Equipartitions to Equidistributions of Colored Point Sets}

In recent years, equipartition problems have come to occupy a central place in geometric and topological combinatorics. Given a fixed integer $r\geq 2$ and a finite collection of finite point sets (or finite continuous measures) in $\R^d$, one seeks a prescribed geometric division of $\R^d$ into $r$ interior disjoint and typically convex regions, each of whose interiors contains no more than $1/r$ of the total number of points from each point set (respectively, of each total measure). The prototypical such result is the classical ham sandwich theorem~\cite{BZ04, ST42}, which asserts that any $d$ finite points sets (or finite continuous measures) in $\R^d$ can be equipartitioned by an affine hyperplane, that is, by the two half-spaces determined by the hyperplane. We refer the reader to the excellent recent survey \cite{RPS22} for more on equipartition theory. 

 Instead of equipartitions of finite point sets in $\R^d$ by full dimensional regions, here we seek their \emph{equidistribution} by \emph{low dimensional} subsets. Specifically, suppose that $A$ is a $k$-flat in $\R^d$, where $1\leq k\leq d-1$. Given an affine hyperplane $H$ in $\R^d$ which intersects but does not contain $A$, we call the intersection $B=A\cap H^+$ of $A$ with a (closed) half-space $H^+$ of $H$ a (closed) \emph{half $k$-flat}. Thus the boundary $\partial B=A\cap H$ of $B$ is a $(k-1)$-flat. For $r\geq 2$, we say that the union $F_r=\cup_{j=1}^r B_j$ of $r$ half $k$-flats $B_1,\ldots, B_r$ is a $k$-dimensional \emph{$r$-fan} provided that there is a single $(k-1)$-flat $C$, called the \emph{center} of the fan, which is the boundary of each $B_j$. A $(d-k)$-dimensional $r$-fan will be said to have codimension $k$. In particular, an $r$-fan of codimension one is the union of $r$ half-hyperplanes centered about an affine space of codimension two. Equipartitions by the regions determined by such $r$-fans have been previously considered (see, e.g.,~\cite{BM01,BM02,Si15}).

 Now suppose that $X$ is a finite set in $\R^d$. We say that an $r$-fan $F_r=\cup_{j=1}^rB_j$ in $\R^d$ \emph{distributes X} if it contains $X$. If $X=X_1\sqcup\cdots\sqcup X_m$ is a partition of $X$ into $m$ non-empty subsets, which we call an \emph{$m$-coloring} of $X$, we say that an $r$-fan \emph{equidistributes} this coloring if it distributes $X$ so that the interior $\Int(B_j)=B_j\setminus C$ of each half-flat contains at most $1/r$ of the points from each $X_i$.

We present equidistribution results for two types of $r$-fans, the first of which is given by certain hyperplane arrangements. Namely, suppose that $H_1,\ldots, H_r$ are oriented affine hyperplanes in $\R^d$ such that (1)  any $r-1$ of the $r$ hyperplanes are independent (i.e., their normal vectors are linearly independent) and (2) the $r$ hyperplanes are dependent and the $r$-fold intersection of the hyperplanes is non-empty (and so is a  flat of codimension $r-1$). Taking indices cyclically, for each $j\in [r]$ we let $A_j=\cap_{i\neq j-1,j} H_i$ be the $(r-2)$-codimensional flat given by intersecting all of the hyperplanes except $H_{j-1}$ and $H_j$, and we let $B_j=H_j^+\cap A_j$ be the half-flat given by intersecting $A_j$ with the half-space $H_j^+$ of $H_j$. The union $F_r=\cup_{j=1}^r B_j$ is then an $r$-fan of codimension $r-2$ whose $(r-1)$-codimensional center is the common intersection of the $H_i$. We shall say that such an $r$-fan is \emph{conical} (See Figure~1 below for a picture when $r=3$). For instance, a conical $r$-fan in $\R^{r-1}$ is a union of $r$ rays emanating from a common point. Our second class of fans are \emph{regular $r$-fans}. These are fans of codimension one for which the dihedral angle between any two successive half-hyperplanes is $2\pi/r$. 
 
 We have the following equidistribution result for conical $r$-fans. Recall that a set of points in $\R^d$ is affinely spanning if it is not contained in an affine hyperplane. 

\begin{theorem} 
\label{thm:real equidistribute}
Let $d, m\geq 1$ be integers, let $r\geq 3$ be a prime power, and let $n\geq (r-1)(d+m+1)+1$. Then any $m$-coloring of a set of $n$ affinely spanning points in $\R^{n-d-1}$ can be equidistributed by a conical $r$-fan. \end{theorem}

 Our proof of Theorem~\ref{thm:real equidistribute} shows that the interior of each half-flat of the conical $r$-fan must contain at least one point from $X$, and, as we show in Section~\ref{sec:generic}, for ``typical'' point sets the union of the interiors of the half-flats contains at least $(r-1)(d+1)+1$ points of $X$. Similar remarks apply to Theorems~\ref{thm:complex equidistribute}, \ref{thm:Dolnikov}, and \ref{thm:colorful} below. As an example, letting $r=3$ and $d=m=1$ in Theorem~\ref{thm:real equidistribute} shows that for a typical set $X$ of $7$ points in $\R^5$ there is a ($4$-dimensional) conical $3$-fan in $\R^5$ which catches all $7$ points of $X$ so that either (a) $1$ point of $X$ lies in the fan's $3$-dimensional center while the remaining $6$ points are equally divided among the fan's $3$ open half-flats, or (b) $2$ points lie in the center while the remaining $5$ are divided so that $2$ of the open half-flats receive $2$ points each and the remaining open half flat receives a single point. Figure~1 gives a picture of case (a) when viewed in the quotient $\R^5/C\cong \R^2$ of $\R^5$ by the fan's center $C$.

\begin{figure}
\label{fig:pic}
    \centering
    \includegraphics[scale=.75]{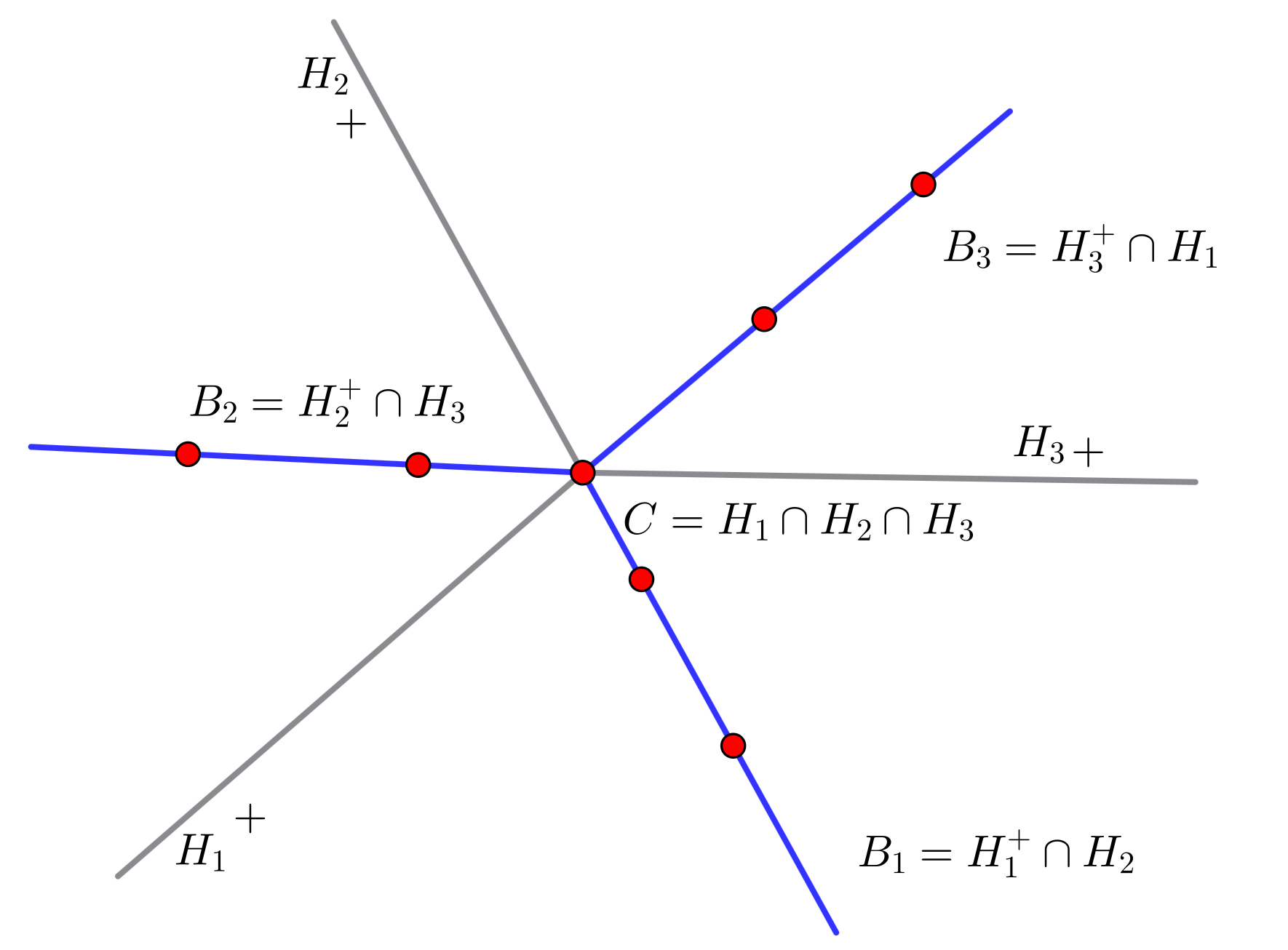}
    \caption{Equidistribution of $7$ points by a conical $3$-fan}
\end{figure}

 Proposition~\ref{prop:optimal} below will show in particular that Theorem~\ref{thm:real equidistribute} is nearly sharp with respect to dimension for all fixed $r\geq 3$ and $m\geq 2$ and so is asymptotically tight. Namely, for each integer $d\geq r-1$ let $n(r,m,d)$ be the maximum integer $n$ such that any $m$-coloring of any affinely spanning set of $n$ points in $\R^d$ can be equidistributed by a conical $r$-fan. Once $d$ is sufficiently large, the lower bound on $n(r,m,d)$ given by Theorem~\ref{thm:real equidistribute} exceeds the exact value by at most four: 

 \begin{corollary}
 \label{cor:values} 
 Let $r\geq 3$ be a prime power and let $m\geq 2$ be an integer. Suppose that $d$ is an integer of the form $d=(r-2)s+t+c$, where $s\geq 1$, $t\in \{0,\ldots, r-3\}$, and $c=(r-1)(m+1)$. Then \[d+s+1\leq n(r,m,d)\leq d+s+5.\]
\end{corollary}
 
When $r=3$, the $r$-fans of Theorem~\ref{thm:real equidistribute} have codimension one but need not be regular. If one considers point sets in $\C^d$, however, then for any prime power $r$ one can guarantee equidistributions by a \emph{complex} regular $r$-fan in $\C^d$, that is a regular $r$-fan in $\R^{2d}$ whose center is a complex affine hyperplane. For this we require that the points complex affinely span $\C^d$, or in other words that they do not all lie on a complex affine hyperplane. See~\cite{SaSo25} for a related complex analogue of the well-known central transversal theorem.

\begin{theorem} 
\label{thm:complex equidistribute}
Let $d,m\geq 1$ be integers, let $r\geq 2$ be a prime power, and let $n\geq (r-1)(2d+m+1)+1$. Then any $m$-coloring of a set of $n$ complex affinely spanning points in $\C^{n-d-1}$ can be equidistributed by a complex regular $r$-fan in $\C^{n-d-1}$.
\end{theorem}

As a regular $2$-fan in $\R^d$ is an affine hyperplane, the $m=2$ case of Theorem ~\ref{thm:complex equidistribute} is closely aligned with the ham sandwich theorem. Namely, let $d=s+m+1$, where $s\geq 1$. Then for any $m$-coloring $X=X_1\sqcup\cdots \sqcup X_m$ of a set $X$ of $d+s+1$ complex affinely spanning points in $\C^d$, there is a real affine hyperplane $H$ in $\C^d$ containing $X$ and a complex affine hyperplane $H_\C$ contained in $H$ such that $H_\C$ -- when viewed as a real affine hyperplane in $H$ -- equipartitions each of the $X_i$. The number of points in $\C^d$ required for this is almost tight for all $m\geq 1$ (see Remark~\ref{rem:complex}).

\subsection{Piercing Distributions}

Recall that a family $\F$ of non-empty subsets of a finite set $X$ in $\R^d$ is said to be \emph{intersecting} if any two members of $\F$ have non-empty intersection. The discrete case of a well-known theorem of Dolnikov~\cite{Do92} states that if $\F_1,\ldots, \F_d$ are intersecting families of non-empty subsets of a finite set in $\R^d$, then there exists an affine hyperplane in $\R^d$ which intersects the convex hull of every member of each family. This result easily implies the ham sandwich theorem (see, e.g.,~\cite{Do92, FS24}). For prime powers $r\geq 3$, we give distribution results which are similar in spirit to Dolnikov's theorem under the weaker intersection criterion that no family contains $r$ pairwise disjoint sets. Instead of an affine hyperplane which pierces the convex hull of every member of each family, we are guaranteed a distributing $r$-fan such that every member of each family is \emph{itself} pierced by two of the fan's half-flats. 

In what follows, it will be convenient to reformulate our intersection condition in terms of hypergraph colorings. Given a finite family $\F$ of non-empty subsets of a set $X$, its \emph{$r$-uniform Kneser hypergraph} $\KG^r(\F)$ is the $r$-uniform hypergraph whose vertices are the elements of $\F$ and whose hyperedges $\{F_1,\ldots, F_r\}$ consist of $r$ pairwise-disjoint elements from $\F$. The chromatic number $\chi(\KG^r(\F))$ of $\KG^r(\F)$ is the minimum number of colors needed to color all the vertices of $\KG^r(\F)$ so that no hyperedge is monochromatic. Thus to say that  $\chi(\KG^r(\F))\leq  m$ means that there are $m$ non-empty subfamilies $\F_1,\ldots,\F_m$ of $\F$ such that $\F=\cup_{j=1}^m\F_i$ and no $\F_i$ contains $r$ pairwise disjoint sets.  

\begin{theorem} 
\label{thm:Dolnikov} Let $d,m \geq 1$ be integers and let $r$ be a prime power. 
\begin{compactenum}[(a)]
\item \label{thm:real Dolnikov} Suppose that $r\geq 3$ and that $X$ is a set of $n\geq (r-1)(d+1)+1$ affinely spanning points in $\R^{n-d-1}$. If $\mathcal{F}$ is a family of non-empty subsets of $X$ with $\chi(\KG^r(\F))\leq m$, then there is a conical $r$-fan which distributes $X$ so that every $A\in \F$ is intersected by at least two closed half-flats of $F_r$.

\item\label{thm:complex Dolnikov} Suppose that $r\geq 2$ and that $X$ is a set of $n\geq (r-1)(2d+m+1)+1$ complex affinely spanning points in $\C^{n-d-1}$. If $\mathcal{F}$ is a family of non-empty subsets of $X$ with $\chi(\KG^r(\F))\leq m$, then there is a complex regular $r$-fan $F_r$ in $\C^{n-d-1}$ which distributes $X$ so that every $A\in \F$ is intersected by at least two closed half-hyperplanes of $F_r$. 
\end{compactenum}
\end{theorem}

The $r=2$ case of Theorem~\ref{thm:Dolnikov}(\ref{thm:complex Dolnikov}) gives an analogue of Dolnikov's theorem for complex instead of real hyperplanes: If $d=s+m+1$ with $s\geq 1$, let $X$ be an affinely spanning set of $d+s+1$ points in $\C^d$ and let $\F_1,\ldots, \F_m$ be non-empty intersecting families of subsets of $X$. Then exists a real affine hyperplane $H$ in $\C^d$ which contains $X$ and a complex affine hyperplane $H_{\C}$ contained in $H$ such that $H_\C$ intersects the convex hull of every member of each $\F_i$. 

\subsection{Proof Scheme}

We provide a brief preview of our proof method, which is a continuation of that used in ~\cite{FS24} to link Dolnikov-type transversal theorems with topological extensions of the classical Radon's theorem~\cite{Ra21}. Following this approach, Theorem~\ref{thm:Dolnikov}, and in turn Theorems~\ref{thm:real equidistribute} and ~\ref{thm:complex equidistribute}, is obtained as a consequence of a topological version of Tverberg's celebrated generalization of Radon's theorem. 

Recall that Tverberg's theorem~\cite{Tv66} asserts that any set of $n\geq (r-1)(d+1)+1$ points in $\R^d$ can be partitioned into $r$ pairwise disjoint sets whose convex hulls have non-empty $r$-fold intersection. The $r=2$ case is Radon's theorem. For prime powers, one has a topological extension of Tverberg theorem~\cite{Oz87, Vo96} which asserts that any continuous map $f\colon \Delta_{n-1}\rightarrow \R^d$ from the $(n-1)$-dimensional simplex to $\R^d$ has an \emph{$r$-Tverberg tuple} -- that is, an $r$-tuple $(\sigma_1,\ldots, \sigma_r)$ of non-empty pairwise disjoint faces of $\Delta_{n-1}$ such that  $f(\sigma_1)\cap\cdots\cap f(\sigma_r)\neq \emptyset$. When $r$ is a prime power, Tverberg's theorem is recovered by letting $f\colon \Delta_{n-1}\rightarrow \R^d$ be a  \emph{linear map}, that is, one which sends each convex combination of vertices of the simplex to the corresponding convex combination of the image of the vertices. In what follows, we identify the vertex set of $\Delta_{n-1}$ with $[n]=\{1,\ldots, n\}$ and we let $\Ver(\sigma)\subseteq [n]$ denote the vertex set of any face $\sigma$ of $\Delta_{n-1}$.

As in ~\cite{FS24}, we will apply the Gale transform, famous for its use in classification of polytopes with few vertices~\cite{Ga56}, to the linear cases of topological Tverberg-type theorems. In particular, Theorem~\ref{thm:Dolnikov} is in essence the Gale dual of the linear version of the following topological Tverberg-type theorem due to Sarkaria~\cite{Sa90, Sa91}. 

\begin{theorem}
\label{thm:Sarkaria} Let $r\geq 2$ be a prime power and let $d,m,n\geq 1$ be integers with $n\geq(r-1)(d+m+1)+1$. Suppose that $\F$ is a family of non-empty subsets of $[n]$ such that $\chi(\KG^r(\F))\leq m$. Then for any continuous map $f\colon \Delta_{n-1}\rightarrow \R^d$, there exists an $r$-Tverberg tuple $(\sigma_1,\ldots, \sigma_r)$ for $f$ such that no element of $\F$ is a subset of $\Ver(\sigma_j)$ for any $j\in [r]$. 
\end{theorem}

As a further example of fan distributions obtainable from Tverberg-type results, we show that the Gale dual of the optimal known results~\cite{BMZ15} for the ``colored Tverberg problem'' gives \emph{rainbow distributions} of a coloring of a finite point set, that is an $r$-fan which distributes the point set so that each of the fan's open half-flats contains at most one point from each color class.

\begin{theorem} 
\label{thm:colorful} 

Let $d\geq 1$ be an integer.  
\begin{compactenum}[(a)] 

\item \label{thm:real colorful} Let $r\geq 4$ be an integer such that $r+1$ is prime and let $X$ be a set of $n\geq r(d+1)$ affinely spanning points in $\R^{n-d-1}$. If $X=X_1\sqcup\cdots\sqcup X_{d+1}$ is a $(d+1)$-coloring with $|X_i|\geq r$ for all $i\in[d+1]$, then there exists a conical $r$-fan in $\R^{n-d-1}$ which rainbow distributes the $(d+1)$-coloring.  

\item \label{thm: complex colorful} Let $r\geq 2$ be an integer such that $r+1$ is prime and let $X$ be a set of $n\geq r(2d+1)$ complex affinely spanning points in $\C^{n-d-1}$. If $X=X_1\sqcup\cdots\sqcup X_{2d+1}$ is a $(2d+1)$-coloring of $X$ such that $|X_i|\geq r$ for all $i\in [2d+1]$, then there exists a complex regular $r$-fan in $\C^{n-d-1}$ which rainbow distributes the $(2d+1)$-coloring.  
\end{compactenum}
\end{theorem}

As an example of Theorem~\ref{thm:colorful}, let $r=4$ and $d=1$. Then for any partition of a set of $8$ affinely spanning points in $\R^6$ into two color classes of size $4$ each, there is a ($4$-dimensional) conical $4$-fan in $\R^6$ whose open half-flats contain at most one point from each color class. If $X$ is typical, then at most one point lies in the fan's $3$-dimensional center.

\subsection{Distributions by Multiple Fans}

Arguably the most fundamental problem in equipartition theory, dating back to Gr\"unbaum~\cite{Gr60} and Hadwiger~\cite{Ha66}, seeks the minimum dimension $d:=\Delta(m,k)$ such that any $m$ finite point sets in $\R^d$ can be equipartitioned by the $2^k$ regions determined by $k\geq 2$ independent affine hyperplanes. While a degrees of freedom argument ~\cite{Ra96} gives the lower bound $\Delta(m,k)\geq \frac{m(2^k-1)}{k}$, very few tight upper bounds are currently known. For two hyperplanes in particular, sharp results exist only when either $m-1$, $m$, or $m+1$ is a power of two~\cite{BFHZ18, BFHZ16,MLVZ06}, in which case $\Delta(m,2)=\lfloor 3m/2\rfloor$. When $r$ is an odd prime, we give analogous extensions of our equidistribution results for two fans, where now given an $m$-coloring $X=X_1\sqcup\cdots\sqcup X_m$ of a finite point set $X$ in $\R^d$ we ask for two $r$-fans $F_r^1=\cup_{i=1}^r B^1_i$ and $F_r^2=\cup_{j=1}^r B^2_j$ in $\R^d$ whose \emph{intersection} contains $X$ so that each intersection of the form $\Int(B_i^1)\cap \Int(B_j^2)$ contains at most $1/r^2$ of the points from each $X_k$. As with the calculations of the tight values of $\Delta(m,2)$, our results are ultimately derived from cohomological techniques which impose restrictions on the values of $m$.

\begin{theorem} 
\label{thm:equidistribute 2} 
Let $d\geq 1$ be an integer, let $r\geq 3$ be a prime number, and let $m\geq 1$ be an integer such that each of the coefficients in the $r$-ary expansion of $m(r-1)/2$ is even. 
\begin{compactenum}[(a)]

\item \label{thm:real equidstirubte 2}
If $n\geq (r-1)(d+1)+\frac{m(r^2-1)}{2}+1$, then any $m$-coloring of a set $X$ of $n$ affinely spanning points in $\R^{n-d-1}$ can be equidistributed by two conical $r$-fans. 

\item \label{thm:complex equidistribute 2}
If $n\geq (r-1)(2d+1)+\frac{m(r^2-1)}{2}+1$, then any $m$-coloring of a set of $X$ of $n$ affinely spanning points in $\C^{n-d-1}$ can be equidistributed by two complex regular $r$-fans in $\C^{n-d-1}$. 
\end{compactenum} 
\end{theorem}

It is straightforward to show (see Remark~\ref{rem:values}) that the $r$-ary expansion of $m(r-1)/2$ has all even coefficients provided $m$ is of the form $m=2a(r^{\ell_1}+\cdots+r^{\ell_k})$ where the $\ell_i\geq 0$ are distinct and $1\leq a\leq r-1$ is odd. In particular, one may let $m=2a(r^\ell-1)/(r-1)$ for any $\ell\geq 1$. When $r\equiv 1~(\text{mod}~4)$, one may also let $m$ be the sum of distinct powers of $r$, and in particular one may let $m=(r^\ell-1)/(r-1)$ for any $\ell\geq 1$.

As was the case for a single fan, Theorem~\ref{thm:equidistribute 2} is the special case of a Dolnikov-type piercing distribution result for two fans (see Theorem~\ref{thm:Dolnikov 2} below) which in turn is the Gale dual of a linear version of a multiple intersection version of Sarkaria's theorem. In essence, this result guarantees two Tverberg tuples which are maximally distinct (see ~\cite{FS24} for the corresponding extension of Theorem~\ref{thm:Sarkaria} when $r=2$).

\begin{theorem}
\label{thm:Sarkaria 2}
Let $d,m,r,n\geq 1$ be integers where $r\geq 3$ is a prime number, the $r$-ary coefficients of $m(r-1)/2$ are all even, and $n\geq (r-1)(d+1)+\frac{m(r^2-1)}{2}+1$. Suppose that $f\colon \Delta_{n-1}\rightarrow \R^d$ is a continuous map and that $\F$ is a family of non-empty subsets of $[n]$ such that $\chi(\KG^{r^2}(\F))\leq m$. Then there exists two Tverberg-tuples $(\sigma_1,\ldots, \sigma_r)$ and $(\tau_1,\ldots, \tau_r)$ for $f$ such that no member of $\F$ is a subset of $\Ver(\sigma_i \cap \tau_j)$ for any $ i,j\in [r]$.
\end{theorem}

Our proof of Theorem~\ref{thm:Sarkaria 2} follows essentially standard methods in topological combinatorics. In this case, the result reduces to a Borsuk-Ulam type statement for the group $\Z_r\oplus \Z_r$ (Proposition~\ref{prop:Borsuk-Ulam}) which we establish using a Chern class argument.\\

The remainder of our paper is organized as follows. In Section~\ref{sec:Gale} we briefly discuss Gale duality in the real and complex setting, which we use in Section~\ref{sec:linear} to establish connections between Tverberg partitions and distributions of point sets by linear fans. This connection, in conjunction with a lifting trick, is used in Section~\ref{sec:Proofs 1D} to prove our distribution results for a single fan. Sections ~\ref{sec:CS-TM} and ~\ref{sec:BU Proof} are devoted to the proof of Theorem~\ref{thm:Sarkaria 2}, which in Section~\ref{sec:Two Fans} is used to prove our distribution results for two fans. Finally, in Section~\ref{sec:optimality and general} we examine typical point sets and prove our sharpness results for equidistributions by a single fan.

\section{Real and Complex Gale Duality}
\label{sec:Gale}

In this section we discuss the properties of the real and complex Gale transform which we will need in translating between Tverberg partitions and fan distributions. These are quite standard in the real setting (see, e.g.,~\cite{Ma12}) and only require the minor adjustment of taking conjugates to hold in the complex case. Nonetheless, for the sake of completeness we shall include their proofs here.

First, we remark that notions of affine dependence and affine span carry over from the real to the complex setting without issue. Thus a sequence of points $a_1,\ldots, a_n$ in $\C^d$ is (complex) affinely dependent if there exists complex scalars $\lambda_1,\ldots,\lambda_n$, not all zero, which simultaneously satisfy the conditions $\sum_{i=1}^n \lambda_i a_i=0$ and $\sum_{i=1}^n \lambda_i=0$. Likewise, the (complex) affine span of $a_1,\ldots, a_n$ is defined to be the set of all affine combinations of the $a_i$, i.e. all the sums of the form $\sum_{i=1}^n\lambda_ia_i$ where $\sum_{i=1}^n \lambda_i=1$. These points affinely span $\C^d$ provided their affine span is all of $\C^d$.

Now let $\K$ denote the field of real or complex numbers and let $a_1,\ldots, a_n$ be a sequence of $n$ points in $\K^d$ which ($\K$-) affinely span $\K^d$ (thus $n\geq d+1$). The Gale transform of this sequence is obtained as follows. First, let $A$ be the $(d+1)\times n$ matrix with coefficients in $\K$ whose column vectors are the $\widehat{a_j}=(a_j,1)\in \K^{d+1}$. Because the $a_j$ affinely span $\K^d$, the $\widehat{a_j}$ linearly span $\K^{d+1}$. Thus $A$ has maximal rank and so its kernel has dimension $n-d-1$. Letting $b_1,\ldots, b_{n-d-1}$ be a basis for $\ker A$, we let $B$ be the $(n-d-1)\times n$ matrix whose $i$-th row is $b_i$, and we let $\overline{B}$ denote its conjugate matrix. The Gale transform of $a_1,\ldots, a_n$ is then defined to be the columns $g_1,\ldots, g_n\in \K^{n-d-1}$ of $\overline{B}$. While the Gale transform is defined only up to a choice of basis for $\ker A$, the basic properties we shall need hold independently of such a choice, and moreover is not difficult to show that a different choice of basis corresponds to a change of the $g_i$ by a linear isomorphism of $\K^{n-d-1}$. 

\begin{proposition}
\label{prop:Gale basic}
Let $g_1,\ldots, g_n$ in $\K^{n-d-1}$ be the Gale transform of a sequence of $n$ affinely spanning points $a_1,\ldots, a_n$ in $\K^d$. Then (1) the $g_i$ linearly span $\K^{n-d-1}$ and (2) $\sum_{i=1}^{n-d-1}g_i=0$. Conversely, any sequence $g_1,\ldots, g_n$ in $\K^{n-d-1}$ satisfying (1) and (2) is the Gale dual of some sequence  $a_1,\ldots, a_n$ of affinely spanning points in $\K^d$.
\end{proposition}

\begin{proof}[Proof of Proposition \ref{prop:Gale basic}]
For (1), note that if $b_1,\ldots, b_{n-d-1}\in \K^{d+1}$ is a basis for $\ker A$, then the conjugate vectors $\overline{b}_1,\ldots, \overline{b}_{n-d-1}$ are linearly independent in $\K^{d+1}$ as well. Thus $\overline{B}$ has maximal rank and so its columns linearly span $\K^{n-d-1}$. For (2), consider the conjugate $\bar{A}$ of $A$. As $Ab_j=0$ for all $j\in [n-d-1]$, $\bar{A}\,\overline{b_j}=0$ for all $j\in [n-d-1]$ as well. Since the last row of $\bar{A}$ is the all 1 vector, $\sum_{j=1}^n g_j=0$. 

Now suppose that $g_1,\ldots, g_n\in \K^{n-d-1}$ satisfies (1) and (2). Let $B$ be the $(n-d-1)\times n$ matrix whose columns are the $g_i$ and let $b_1,\ldots, b_{n-d-1}$ be its rows. As the $g_i$ linearly span $\K^{n-d-1}$, so do the $\overline{g_i}$ and therefore $\overline{B}$ has maximum rank. Thus $\dim\ker \overline{B}=d+1$. Since $\sum_i g_i=0$, $\sum_i \overline{g_i}=0$ as well and so one can choose a basis $w_1,\ldots, w_{d+1}$ for $\ker \overline{B}$ with $w_{d+1}$ the all one vector. Now define $A$ to be the matrix whose rows are the $w_i$. As $A$ has full rank its columns span $\K^{d+1}$, and, since the last row of $A$ is the all one vector, the columns of $A$ are of the form $\widehat{a_1}=(a_1,1),\ldots, \widehat{a_n}=(a_n,1)$. Thus $a_1,\ldots, a_n$ affinely span $\K^d$. By construction we have $\text{Row}\,A=\ker \overline{B}$, so $\text{Row}\,\overline{A}=\ker B$  and thus $\ker A=(\text{Row}\,\overline{A})^\perp=(\ker B)^\perp=\text{Row}\, \overline{B}.$ Hence $\overline{b}_1,\ldots, \overline{b}_{n-d-1}$ is a basis for $\ker A$, and so the columns $g_1,\ldots, g_n$ of $\overline{\overline{B}}=B$ are the Gale transform of $a_1,\ldots, a_n$. \end{proof}

For our purposes, the essential property of the Gale transform is that it gives a correspondence between affine dependencies among the $a_1,\ldots, a_n$ and allocations of $g_1,\ldots, g_n$ with respect to a codimension one linear subspace. More precisely, let $\langle u, w\rangle=\sum_{i=1}^n u_i\overline{w}_i$ denote the standard Hermitian inner product of two vectors $u$ and $w$ in $\K^n$. One then has the following.

\begin{proposition}
\label{prop: Gale dependence} Let $g_1,\ldots, g_n \in \K^{n-d-1}$ be the Gale transform of a sequence $a_1,\ldots, a_n$ of affinely spanning points in $\K^d$. If $\lambda=(\lambda_1,\ldots, \lambda_n)$ is a non-zero vector in $\K^n$ such that (1) $\sum_{i=1}^n\lambda_ia_i=0$ and (2) $\sum_{i=1}^n\lambda_i=0$, then there exists a non-zero vector $\alpha$ in $\K^{n-d-1}$ such that $\langle \alpha, g_i \rangle =\lambda_i$ for all $i\in [n-d-1]$. Conversely, given a non-zero vector $\alpha$ in $\K^{n-d-1}$, define $\lambda_i=\langle \alpha, g_i \rangle$ for all $i\in [n]$. Then $\lambda=(\lambda_1,\ldots, \lambda_n)\in \K^n$ is non-zero, $\sum_{i=1}^n\lambda_ia_i=0$, and $\sum_{i=1}^n\lambda_i=0$. 
\end{proposition}

\begin{proof}[Proof of Proposition~\ref{prop: Gale dependence}] Let $A=\begin{bmatrix} \widehat{a_1}| & \cdots & |\widehat {a_n}\end{bmatrix}$ be the matrix associated to the $a_i$ in the construction of the Gale transform and let $b_1,\ldots, b_{n-d-1}$ be a chosen basis for $\ker A$. First, suppose that $\lambda=(\lambda_1,\ldots, \lambda_n)\in \K^{n}\setminus \{0\}$ satisfies (1) and (2). Thus $\lambda\in \ker A$, and as $\lambda\neq 0$ there exists some $\alpha=(\alpha_1,\ldots, \alpha_{n-d-1})\in \K^{n-d-1}\setminus\{0\}$ such that $\lambda=\sum_{j=1}^{n-d-1} \alpha_j b_j$. Taking the complex conjugate gives $\overline{\lambda_i}=\sum_{j=1}^{n-d-1}(\overline{b_j})_i\overline{\alpha_j}$ for all $i\in [n]$. But $(g_i)_j=(\overline{b_j})_i$, so $\overline{\lambda_i}=\langle g_i,\alpha\rangle$ and hence $\lambda_i=\langle \alpha, g_i\rangle$ as claimed. Conversely, suppose that $\alpha=(\alpha_1,\ldots, \alpha_{n-d-1})\in \K^{n-d-1}\setminus\{0\}$. Defining $\lambda_i=\langle \alpha, g_i\rangle$ for all $i\in [n]$ and $\lambda=(\lambda_1,\ldots, \lambda_n)$, it follows easily from reversing the steps just given that $\lambda=\sum_{j=1}^{n-d-1}\alpha_jb_j$. Thus $\lambda\in \ker A$ and so (1) and (2) are satisfied. Finally, $\lambda\neq 0$ because the $b_j$ form a basis for $\ker A$.\end{proof}

\section{Tverberg Partitions and Linear Fans}
\label{sec:linear}

 In this section, we describe an explicit connection between Tverberg $r$-tuples and \emph{linear} $r$-fans, i.e, $r$-fans whose center is a linear subspace. This connection will be used crucially in the proofs of Theorems~\ref{thm:real equidistribute},~\ref{thm:complex equidistribute},~\ref{thm:Dolnikov}, and~\ref{thm:colorful} in Section~\ref{sec:Proofs 1D}.

Let $a_1,\ldots, a_n$ be a sequence in $\R^d$ and let $f\colon \Delta_{n-1}\rightarrow \R^d$ be the corresponding linear map determined by setting $f(i)=a_i$ for all $i\in [n]$. Supposing that $(\sigma_1,\ldots, \sigma_r)$ is a Tverberg $r$-tuple for $f$, we set $I_j=\Ver(\sigma_j)$ for all $j\in [r]$. We say that the Tverberg-tuple is \emph{proper} if there exist non-negative real scalars $t_1,\ldots, t_n$ such that 
\[\sum_{i\in I_1}t_ia_i=\cdots=\sum_{i\in I_r}t_ia_i,\] where (1) $t_i>0$ for all $i\in \sqcup_{j=1}^r I_j$ and $\sum_{i\in I_j}t_i=1$ for all $j\in [r]$, and (2) $t_i=0$ for all $i\notin\sqcup_{j=1}^r I_j$.  

\subsection{Linear Conical Fans}

Given a sequence $x_1,\ldots, x_n$ of $n$ points in $\R^d$ and $r$ non-empty pairwise disjoint subsets $I_1,\ldots, I_r$ of $[n]$, we say an $r$-fan $F_r=\cup_{j=1}^r B_j$ in $\R^d$ with center $C$ \emph{distributes the $x_i$ according to the $I_j$} if 
\begin{compactenum}
        \item $x_i\in \Int(B_j)$ for all $i\in I_j$ and all $j\in [r]$, and
    \item $x_i\in C$ for all $i\notin \sqcup_{j=1}^r I_j$. 
\end{compactenum}

Gale duality gives an equivalence between proper $r$-Tverberg partitions of sequences of points in $\R^d$ and distributions of sequences of points by linear conical $r$-fans. 

\begin{lemma}
\label{lem:linear conical fans} 

Let $d\geq 1$ and $r\geq 3$ be integers. Suppose that $a_1,\ldots, a_n$ is a sequence of affinely spanning points in $\R^d$ with Gale transform $g_1,\ldots, g_n$ in $\R^{n-d-1}$, and let $f\colon \Delta_{n-1}\rightarrow \R^d$ be the affine linear map such that $f(i)=a_i$ for all $i\in [n]$.

\begin{compactenum}[(a)] 
\item Suppose that $(\sigma_1,\ldots, \sigma_r)$ is a proper $r$-Tverberg tuple for $f$ and let $I_1,\ldots, I_r\subset [n]$ be the corresponding vertex sets of the $\sigma_j$. Then there exists a linear conical $r$-fan $F_r$ in $\R^{n-d-1}$ which distributes the $g_i$ according to the $I_j$. 

\item Suppose that $I_1,\ldots, I_r$ are non-empty pairwise disjoint subsets of $[n]$ and that $F_r$ is a linear conical $r$-fan in $\R^{n-d-1}$ which distributes the $g_i$ according to the $I_j$. 
Then there exists a proper $r$-Tverberg tuple $(\sigma_1,\ldots, \sigma_r)$ for $f$ such that $\Ver(\sigma_j)=I_j$ for all $j\in[r]$. 
\end{compactenum}
\end{lemma}

\begin{proof}[Proof of Lemma~\ref{lem:linear conical fans}]

For part (a), suppose that $\sum_{i\in I_1}t_ia_i=\sum_{i\in I_2}t_ia_i=\cdots=\sum_{i\in I_r} t_ia_i$, where $\sum_{i\in I_j}t_i=1$ for all $j\in [r]$, $t_i>0$ for all $i\in \sqcup_{j=1}^r I_j$, and $t_i=0$ for all $i\notin\sqcup_{j=1}^r I_r$. For each $j\in [r]$, define $\lambda_j\in \R^n\setminus\{0\}$ by $(\lambda_j)_i=t_i$ if $i\in I_j$, $(\lambda_j)_i=-t_i$ if $i\in I_{j-1}$, and $(\lambda_j)_i=0$ otherwise. Thus $\sum_{i=1}^n (\lambda_j)_i=0$ and $\sum_{i=1}^n (\lambda_j)_i a_i=0$.  By Proposition~\ref{prop: Gale dependence}, there is some non-zero vector $\alpha_j\in \R^{n-d-1}$ such that \begin{equation}
\label{eqn:inner product} \langle \alpha_j, g_i\rangle=(\lambda_j)_i\end{equation} for all $i\in [n]$. 

We claim that any $r-1$ of the $\alpha_j$ are linearly independent, but that $\alpha_1,\ldots, \alpha_r$ are not. To see this, observe that since the $g_i$ linearly span $\R^{n-d-1}$, equation~(\ref{eqn:inner product}) implies that any linear relation $\sum_{j=1}^{n-d-1}\mu_j\alpha_j=0$ holds among the $\alpha_j$ if and only if the corresponding relation $\sum_{j=1}^{n-d-1}\mu_j\lambda_j=0$ holds among the $\lambda_j$. It is easily seen that any $r-1$ vectors among the $\lambda_j$ are linearly independent, so any $r-1$ of the $\alpha_j$ are as well. Similarly, $\sum_{j=1}^r \lambda_j=0$ implies that $\sum_{j=1}^r \alpha_j=0$. For each $j\in [r]$, let $H_j=\langle \alpha_j\rangle^\perp$ be the linear hyperplane corresponding to $\alpha_j$. Now let $A_j=\cap_{i\neq j,j-1} H_i$ and $B_j= H_j^+ \cap A_j$, where $H_j^+=\{x\in \R^{n-d-1}\mid \langle \alpha_j, x\rangle\geq 0\}$. Thus $H_j\cap A_j=\cap_{i=1}^r H_i$ for all $j\in [r]$ and $F_r=\cup_{j=1}^r B_j$ is a linear conical $r$-fan. The distribution claim concerning the $g_i$ also follows from equation~(\ref{eqn:inner product}). For if $j\in [r]$ and $i\in I_j$, we have $\langle \alpha_j,g_i\rangle = (\lambda_j)_i=t_i>0$, while $\langle \alpha_k, g_i\rangle=(\lambda_k)_i=0$ if $k\neq j-1, j$. Thus $g_i\in \Int(B_j)=\Int(H_j^+)\cap A_j$. On the other hand, $\langle g_i,\alpha_j\rangle=0$  for all $j\in [r]$ if $i\notin \sqcup_{j=1}^r I_j$ and therefore $g_i\in C$. 

For part (b), suppose that $F_r$ is a linear conical $r$-fan in $\R^{n-d-1}$ determined by the hyperplanes $H_1=\langle \alpha_1\rangle^\perp ,\ldots,H_r=\langle \alpha_r \rangle^\perp$ in $\R^{n-d-1}$. By assumption, $\alpha_1,\ldots, \alpha_r$ are linearly dependent, so $\sum_{j=1}^r \mu_j\alpha_j=0$ for some scalars $\mu_1,\ldots, \mu_r$, not all of which are zero. Since any $r-1$ of the $\alpha_j$ are linearly independent, $\mu_j\neq 0$ for all $j\in [r]$. Thus by scaling $\alpha_j$ by $\mu_j$ we may assume that $\sum_{j=1}^r \alpha_j=0$. Without loss of generality, we may let $H^+_j=\{x\in \R^{n-d-1}\mid \langle \alpha_j, x\rangle \geq 0\}$ for all $j\in [r]$.  

For each $j\in [r]$, define $\lambda_j=(\lambda_{j,1},\ldots, \lambda_{j,n})\in \R^n$ by $\lambda_{j,i}=\langle \alpha_j, g_i\rangle$ for all $i\in [n]$.  As $\alpha_j\neq 0$, by Proposition~\ref{prop: Gale dependence} we have that $\lambda_j\neq 0$ and that $\sum_{i=1}^n \lambda_{j,i}=0$ and $\sum_{i=1}^n \lambda_{j,i}a_i=0$. Now let $j\in [r]$ and suppose that $i\in I_j$. Since $g_i\in \Int(B_j)$, we have $\lambda_{j,i}>0$ and $\lambda_{k,i}=0$ for all $k\neq j-1, j$. As $\sum_{j=1}^r\alpha_j=0$, we have $\langle \sum_{j=1}^r\alpha_j, g_i\rangle =\langle \alpha_{j-1}, g_i\rangle +\langle \alpha_j,g_i\rangle=0$. Thus $\lambda_{j-1,i}=-\lambda_{j,i}$ for all $i\in I_j$. Again fix $j\in [r]$ and consider the relation $\sum_{i=1}^n \lambda_{j,i}a_i=0$. If $i\notin I_j\sqcup I_{j+1}$, then $g_i\notin \Int(B_j)\sqcup \Int(B_{j+1})$. In particular, $g_i\in \cap_{\ell\neq k,k-1} H_\ell$ for some $k\neq j,j+1$ and so $\langle \alpha_j,g_i\rangle=0$ for all $i\notin I_j\sqcup I_{j+1}$. Thus $\sum_{i\in I_j}\lambda_{j,i}a_i+\sum_{i\in I_{j+1}}\lambda_{j,i}a_i=0$ and so $\sum_{i\in I_j}\lambda_{j,i}a_i=-\sum_{i\in I_{j+1}}\lambda_{j,i}a_i=\sum_{i\in I_{j+1}}\lambda_{j+1,i}a_i$ for any $j\in [r]$. We therefore have that  \begin{equation}
    \label{eqn:equal}
\sum_{i\in I_1} \lambda_{1,i}a_i=\cdots=\sum_{i\in I_r}\lambda_{r,i}a_i.\end{equation} 

For each $j\in [r]$ we let $s_j=\sum_{i\in I_j}\lambda_{j,i}$. As $\langle \alpha_j, g_i\rangle>0$ for all $i\in I_j$, we have $s_j>0$. Letting $t_i=\lambda_{j,i}/s_j$ for each  $i\in I_j$, we have that $t_i>0$ for all $i\in I_j$ and that $\sum_{i\in I_j} t_i=1$. Fixing $j\in [r]$ and considering the sum $\sum_{i=1}^n \lambda_{j,i}=0$, the identical reasoning to that used in deriving equation~(\ref{eqn:equal}) shows that $\sum_{i\in I_j}\lambda_{j,i}=\sum_{i\in I_{j+1}}\lambda_{j+1,i}$ for all $j\in [r]$ and so that $s_1=\cdots=s_r$. Dividing equation~(\ref{eqn:equal}) by $s_1$, we now have $\sum_{i\in I_1} t_i a_i=\cdots=\sum_{i\in I_r}t_ia_i$. This gives a proper Tverberg tuple for the linear map $f$.  
\end{proof}

\subsection{Complex Linear Fans}

Before giving a complex analogue of Lemma~\ref{lem:linear conical fans}, we first give a description of complex regular $r$-fans in $\C^d$ in terms of the standard Hermitian inner product. Any complex affine hyperplane in $\C^d$ is of the form $H_\C:=\{z\in \C^d\mid \langle \alpha, z\rangle =\beta\}$ for some $(\alpha,\beta) \in \C^n\setminus\{0\}\times \C$. Let $\omega_r=\exp(2\pi i/r)$ be the standard $r$-th root of unity. If $F_r=\cup_{j=1}^rB_j$ is a complex regular $r$-fan centered about $H_\C$, then there exists some unit complex number $\lambda$ such that $B_j=\{z\in \C^d\mid \langle \alpha, z\rangle =\beta+t\lambda\omega_r^j\,\,\text{for some}\,\, t\geq 0\}$ for all $j\in [r]$. Moreover, one may take $\lambda=1$ after scaling $(\alpha,\beta)$ by $\lambda^{-1}$.

\begin{lemma}
\label{lem: Linear Complex Fans} Let $d\geq 1$ and $r\geq 2$ be integers. Suppose that $a_1,\ldots, a_n$ are complex affinely spanning points in $\C^d$ whose Gale transform is $g_1,\ldots, g_n\in \C^{n-d-1}$, and let $f\colon \Delta_{n-1}\rightarrow \C^d\cong \R^{2d}$ be the (real) linear map given by setting $f(i)=a_i$ for all $i\in [n]$. Suppose that $(\sigma_1,\ldots, \sigma_r)$ is a proper $r$-Tverberg tuple for $f$ and let $I_1,\ldots, I_r\subset [n]$ be the corresponding vertex sets of the $\sigma_j$. Then there exists a linear complex regular $r$-fan $F_r$ in $\C^{n-d-1}$ which distributes the $g_i$ according to the $I_j$. 
\end{lemma}

\begin{proof}[Proof of Lemma~\ref{lem: Linear Complex Fans}] As before, suppose that  $I_1,\ldots, I_r\subset [n]$ are pairwise disjoint non-empty sets such that $\sum_{i\in I_1}t_ia_i=\cdots=\sum_{i\in I_r}t_ia_i$ where $t_i>0$ for all $i\in \sqcup_{j=1}^r I_j$ and $\sum_{i\in I_j}t_i=1$ for all $j\in [r]$, and moreover that $t_i=0$ for all $i\notin\sqcup_{j=1}^rI_j$. To construct a complex affine dependence among the $a_i$,  define $\lambda=(\lambda_1,\ldots, \lambda_n)\in \C^n\setminus\{0\}$ as follows: for each $j\in [r]$ and each $i\in I_j$,  let $\lambda_i=t_i\omega_r^j$, and let $\lambda_i=0$ for all $i\notin \sqcup_{j=1} I_j$. As  $\sum_{i\in I_j} t_j=1$ for all $j\in [r]$, we have $\sum_{i=1}^n \lambda_i=\sum_{j\in [r]}\sum_{i\in I_j} t_i\omega_r^j=\sum_{j\in [r]}\omega_j^r=0$. Likewise, we have $\sum_{i\in I_j}t_ia_i=c$ is constant over $j\in [r]$ and so $\sum_{i\in [n]}\lambda_i a_i=\sum_{j\in [r]}\sum_{i\in I_j}t_i\omega_r^ja_i=c\sum_{j\in [r]}\omega_r^j=0$ as well. By Proposition~\ref{prop: Gale dependence}, there is some $\alpha\neq 0$ such that $\langle \alpha, g_i\rangle =\lambda_i$ for all $i\in [n]$. Let $H_{\mathbb{C}}=\langle \alpha\rangle^\perp$ be the resulting linear complex hyperplane in $\C^{n-d-1}$, and for each $j\in [r]$ let $B_j=\{z\in \C^{n-d-1}\mid \langle \alpha, z\rangle = t\omega_r^j\,\,\text{for some}\,\, t\geq 0\}$. Thus $F_r=\cup_{j=1}^r B_j$ is a complex regular $r$-fan $F_r$ centered about $H_{\mathbb{C}}$. By construction, $g_i\in \Int{B_j}$ for all $i\in I_j$ and any $j\in [r]$, while $g_i\in H_{\mathbb{C}}$ if $i\notin \sqcup_{j\in [r]}I_j$.
\end{proof}

\begin{remark}
\label{rem:equivalence} 
We observe that distributions by complex linear $r$-fans are equivalent to proper Tverberg partitions when $r=2$. To see this, suppose that $H=B_1\cup B_2$ is a real hyperplane in $\C^{n-d-1}$ centered about a complex linear hyperplane $H_{\C}=\langle \alpha\rangle^\perp$ and that there exist two disjoint non-empty subsets $I_1,I_2$ of $[n]$ such that $g_i\in \Int{B_j}$ for all $i\in I_j$ and all $j\in \{1,2\}$, while $g_i\in H_{\mathbb{C}}$ otherwise. Thus $\langle \alpha, g_i\rangle =t_i(-1)^j$ for all $i\in I_j$ and all $j\in\{1,2\}$, where $t_i>0$ for all $i\in I_1\sqcup  I_2$, while $\langle \alpha, g_i \rangle =0$ if $i\notin I_1\sqcup I_2$. By Proposition~\ref{prop: Gale dependence}, we have $\sum_{i\in I_1}t_ix_i=\sum_{i\in I_2}t_ix_i$ and $\sum_{j\in I_1}t_i=\sum_{j\in I_2} t_i$. Letting $t_i'=\frac{t_i}{\sum_{j\in I_1}t_i}$ for all $i\in I_1\sqcup I_2$, we have that $t_i'>0$ for all $i\in I_1\sqcup I_2$, $\sum_{i\in I_1}t_i'=\sum_{i\in I_2}t_i'$, and $\sum_{i\in I_1} t_i'x_i=\sum_{i\in I_2} t_i'x_i$. 
\end{remark}

\section{Distributions by a Single Fan}  
\label{sec:Proofs 1D}

 In this section we prove our distribution results for a single fan, each of which follows by applying Gale duality to the linear version of an associated topological Tverberg-type theorem.

\subsection{Proof of Theorems~\ref{thm:real equidistribute} and ~\ref{thm:complex equidistribute}}

We begin by proving our equidistribution results for a single fan, which are a special case of Theorem~\ref{thm:Dolnikov}. 

\begin{proof}[Proof of Theorems~\ref{thm:real equidistribute} and ~\ref{thm:complex equidistribute}] Let $\K=\R$ or $\C$. We let $n\geq (r-1)(d+m+1)$ when $\K=\R$ and $n\geq (r-1)(2d+m+1)+1$ when $\K=\C$. 
Suppose that $X=X_1\sqcup\cdots\sqcup X_m$ is an $m$-coloring of a set of $n$ affinely spanning points in $\K^{n-d-1}$. Corresponding to each $X_i$, define a family of non-empty subsets of $X$ by $\F_i=\{F\subseteq X\mid |F\cap X_i|>|X_i|/r\}$ and let $\F=\cup_{i=1}^m \F_i$. As no $\F_i$ can contain $r$ pairwise disjoint members by the pigeon-hole principle, $\chi(\KG^r(\F))\leq m$.  Theorem~\ref{thm:Dolnikov} now guarantees an $r$-fan $F_r=\cup_{j=1}^r B_j$ in $\K^{n-d-1}$ (a conical $r$-fan when $\K=\R$, and a complex regular fan when $\K=\C)$ such that no $\Int(B_j)$ contains any element of $\F$. It follows that $|\Int(B_j)\cap X_i|\leq |X_i|/r$ for each $j\in [r]$ and $i\in[m]$, so that $F_r$ is the desired $r$-fan. Indeed, were $|\Int(B_j)\cap X_i|>|X_i|/r$ for some $j\in [r]$ and $i\in [m]$ then $\Int(B_j)\cap X$ would be an element of $\F_i$ and so $\Int(B_j)\cap X$ would not be a subset of $\Int(B_j)$. 
\end{proof}

\subsection{Proof of Theorem~\ref{thm:Dolnikov}} We now prove our piercing distribution result for a single fan.

\begin{proof}[Proof of Theorem~\ref{thm:Dolnikov}] Again let $\K=\R$ or $\C$ and let $X=\{x_1,\ldots, x_n\}$ be a set of $n$ affinely spanning points in $\K^{n-d-1}$, where $n\geq (r-1)(d+m+1)+1$ when $\K=\R$ and $n\geq (r-1)(2d+m+1)+1$ if $\K=\C$. We are given that $\F$ is a family of non-empty subsets of $X$ satisfying $\chi(\KG^r(\F))\leq m$. We now lift $X$ to the $\K$-affine hyperplane $\K^{n-d-1}\times \{1\}$ in $\K^{n-d}$. The sequence $g_1=(x_1,1),\ldots, g_n=(x_n,1)$ in $\K^{n-d}$ is linearly spanning, and including $g_{n+1}=-\sum_{i=1}^n g_i$ gives a sequence of $n+1$ points in $\K^{n-d}$ which both linearly span and which sum to zero. By Proposition~\ref{prop:Gale basic}, the sequence $g_1,\ldots, g_{n+1}$ is the Gale transform of $n+1$ points $a_1,\ldots, a_{n+1}$ in $\K^d$. We let $f\colon \Delta_n\rightarrow \K^d$ be the (real) linear map determined by sending each $i\in [n+1]$ to $a_i$ (here we identify $\C^d$ with $\R^{2d}$).

Corresponding to $\F$, let $\F'$ be the family of all subsets $\{i_1,\ldots, i_k\}$ of $[n]$ such that $\{x_{i_1},\ldots, x_{i_k}\}$ is in $\F$. As $\chi(\KG^r(\F))\leq m$, $\chi(\KG^r(\F'))\leq m$ as well. Given the value of $n$, we may apply Theorem~\ref{thm:Sarkaria} to the restriction $f\mid_{\Delta_{n-1}}\colon \Delta_{n-1}\rightarrow \K^d$ of $f$ to the simplex $\Delta_{n-1}$ whose vertex set is $[n]$. Thus there is a proper Tverberg tuple $(\sigma_1,\ldots, \sigma_r)$ for $f|_{\Delta_{n-1}}$ such that each $I_j=\Ver(\sigma_j)\subset [n]$ does not contain any element of $\F'$. As $(\sigma_1,\ldots, \sigma_r)$ is also a proper $r$-Tverberg partition for the original map $f$, by Lemmas~\ref{lem:linear conical fans} and ~\ref{lem: Linear Complex Fans}) there is a linear $r$-fan $F_r=\cup_{j=1}^rB_j$ in $\K^{n-d}$ (conical when $\K=\R$, and complex linear when $\K=\C$) which distributes $g_1,\ldots, g_{n+1}$ according to the $I_j$. Thus $g_i\in \Int(B_j)$ for all  $i\in I_j$ and all $j\in [r]$, while $g_i\in C=\cap_{j=1}^r H_j$ for all $i\notin \sqcup_{j=1}^r I_j$. As $n+1$ does not lie in any of the $I_j$, we have that $g_{n+1}$ lies in $C$. 

Now let $\K=\R$. For each $j\in [r]$ we have $H_j=\langle \alpha_j\rangle^\perp$ where $\alpha_j=(\beta_j, \gamma_j)\in \R^{n-d-1}\times \R$ is non-zero. The linear conical $r$-fan is of the form $F_r=\cup_{j=1}^r B_j$, where $B_j=H_j^+\cap A_j$, $A_j=\cap_{i\neq j,j-1} H_i$, and $H_j^+$ is some half-space determined by $H_j$. For each $j\in [r]$, let $H'_j=H_j\cap (\R^{n-d-1}\times \{1\})$ be the intersection of $H_j$ with the plane of height one in $\R^{n-d}$. Thus $H'_j=\{(x,1)\in \R^{n-d-1}\times \{1\}\mid \langle \beta_j, x\rangle =-\gamma_j\}$. Likewise, let $B'_j=B_j\cap (\R^{n-d-1}\times \{1\})$ for all $j\in [r]$. Letting $H_j''$ and $B_j''$ be the corresponding projections of $H_j'$ and $B_j'$ onto $\R^{n-d-1}$, we will show that $F''_r=\cup_{j=1}^r B_j''$ is a conical $r$-fan in $\R^{n-d-1}$. To that end, we first show that any $r-1$ of $\beta_1,\ldots, \beta_r$ are linearly dependent in $\R^{n-d-1}$. In particular, each $\beta_j$ is non-zero, so each $H''_j$ is an affine hyperplane in $\R^{n-d-1}$. Without loss of generality, consider $\beta_1,\ldots, \beta_{r-1}$ and suppose that $\sum_{j=1}^{r-1} \mu_j \beta_j=0$. We have $\langle \sum_{j=1}^{r-1}\mu_j\alpha_j, g_{n+1}\rangle=0$ because $g_{n+1}\in H_j$ for all $j\in [r]$. As $g_{n+1}=(-x, -n)$ where $x=\sum_{i=1}^n x_i\in \R^{n-d-1}$, we have $\langle \sum_{j-1}^r\mu_j\alpha_j, g_{n+1}\rangle = -n\sum_{j=1}^{r-1} \mu_j\gamma_j$. Thus $\sum_{j=1}^{r-1}\mu_j\gamma_j=0$, and since $\sum_{j=1}^{r-1}\mu_j\beta_j=0$ by assumption we have that $\sum_{j=1}^{r-1}\mu_j\alpha_j=0$ as well. However, any $r-1$ of the $\alpha_j$ are linearly independent in $\R^{n-d}$, so each $\mu_j$ is zero and $\beta_1,\ldots, \beta_{r-1}$ are linearly independent. On the other hand, $\beta_1,\ldots, \beta_r$ are linearly dependent because $\alpha_1,\ldots, \alpha_r$ are. Finally, we show that $\cap_{j=1}^r H'_j\neq\emptyset$ and so that $H_j''\neq \emptyset$ as well. But $g_{n+1}=(-x,-n) \in \cap_{j=1}^r H_j$, so $(\frac{x}{n},1)\in \cap_{j=1}^r H_j'\neq \emptyset$. 

In the complex setting, we have a complex regular $r$-fan $F_r=\cup_{j=1}^r B_j$ in $\C^{n-d}$ centered about a linear complex hyperplane $H_\C$. As above, we let $H_{\mathbb{C}}'=H_{\mathbb{C}}\cap (\C^{n-d-1}\times \{1\})$, $B_j'=B_j\cap (\C^{n-d-1}\times \{1\})$ for each $j\in [r]$, and we let $H''_\C$ and $B''_j$ be their corresponding projections onto $\C^{n-d-1}$.  We claim that $F_r''=\cup_{j=1}^r B''_j$ is a regular $r$-fan in $\C^{n-d-1}$ centered about the complex affine hyperplane $H''_{\C}$. To verify this, we have $H_{\C}=\langle \alpha \rangle^\perp$, where $\alpha=(\beta, \gamma)\in \C^{n-d-1}\times \C$ is non-zero. Thus $H''_{\C}=\{z\in \C^{n-d-1}\mid \langle \beta, z\rangle =-\gamma\}$ and $B''_j=\{(z\in \C^{n-d-1}\mid \langle \beta, z\rangle =-\gamma+t\omega_r^j\,\,\text{for some}\,\, t\geq 0\}$. We therefore have that $H_{\C}''$ and $F_r''$ are as claimed provided $\beta\neq 0$. But $g_{n+1}=(-x,-n)\in H_\C$, where $x=\sum_{i=1}^n x_i$, and so $\langle \beta, x\rangle=-n\gamma$. If $\beta=0$, then $\gamma=0$ and therefore $\alpha=0$ as well, a contradiction. Thus $\beta\neq 0$ and $F''_r$ is indeed a complex regular $r$-fan with center $H_\C''$.  

Again let $\K=\R$ or $\C$. Since $g_1,\ldots, g_{n+1}$ are distributed according to the $I_j$, for each $i\in [n]$ we have (1) for each $j\in [r]$ that $x_i$ lies in $\Int(B''_j)$ if $i\in I_j$ and (2) that $x_i$ lies in the center of $F_r''$ for all $i\notin\sqcup_{j=1}^r I_j$. In particular, $F''_r$ distributes $X$. To complete the proof, we show that no member of $\F$ is contained in any $\Int(B''_j)$. So let $j\in [r]$, let $A\in \F$, and let $A'$ be the corresponding element of $\F'$. As $A'$ is not contained in $I_j$, there exists some $i\in [n]$ such that $i\in A'\setminus I_j$. Thus $x_i\in A\setminus\Int(B_j'')$ and so $A$ is not contained in $\Int(B''_j)$.
\end{proof}

\subsection{Proof of Theorem~\ref{thm:colorful}}

We conclude this section by proving our rainbow distribution result. This follows the same procedure as before, except now we substitute Sarkaria's theorem with the following result~\cite[Corollary 2.3]{BMZ15} which gives optimal results to the ``colored Tverberg'' problem of B\'ar\'any and Larman~\cite{BL92}.

\begin{theorem}
\label{thm:Optimal Colored} Let $d\geq 1$ be an integer, let $r\geq 2$ be an integer such that $r+1$ is prime, and let $n\geq r(d+1)$. If $f\colon \Delta_{n-1}\rightarrow \R^d$ is a continuous map and $C_1\sqcup\cdots \sqcup C_{d+1}$ is a $(d+1)$-coloring of the vertex set of $\Delta_n$ such that $|C_i|\geq r$ for all $i\in[d+1]$, then there exists a Tverberg $r$-tuple $(\sigma_1,\ldots, \sigma_r)$ for $f$ such that $|\Ver(\sigma_j)\cap C_i|\leq 1$ for all $j\in [r]$ and all $i\in[d+1]$.
\end{theorem} 

\begin{proof}[Proof of Theorem~\ref{thm:colorful}] Let $\K=\R$ or $\C$. We let $n\geq r(d+1)$ if $\K=\R$ and $n\geq r(2d+1)$ if $\K=\C$, and we suppose that $X=\{x_1,\ldots, x_n\}$ is a set of $n$ affinely spanning points in $\K^{n-d-1}$. In the real case we have that $X=X_1\sqcup\cdots\sqcup X_{d+1}$ is a $(d+1)$-coloring with $|X_k|\geq r$ for all $k\in [d+1]$, and likewise in the complex case with $d+1$ replaced by $2d+1$. As in the proof of Theorem~\ref{thm:Dolnikov}, lift $X$ to the $\K$-affine hyperplane $\K^{n-d-1}\times\{1\}$ in $\K^{n-d}$. As there, the points $g_1=(x_1,1)\,\ldots, g_n=(x_n,1), g_{n+1}=-\sum_{i=1}^n g_i$ are the Gale transform of some sequence $a_1,\ldots, a_{n+1}$ of affinely spanning points in $\K^d$. As before, we let $f\colon \Delta_n\rightarrow \K^d$ be the (real) affine linear map defined by $f(i)=a_i$ for all $i\in [n+1]$. If $\K=\R$, we define $C_k=\{i\in [n] \mid x_i\in X_k\}$ for each $k\in [d+1]$. Thus $C_1\sqcup \cdots \sqcup C_{d+1}$ is a $(d+1)$-coloring of the vertex set of $\Delta_{n-1}$ into $d+1$ color classes, each of which has size at least $r$. We proceed analogously in the complex setting. Applying Theorem~\ref{thm:Optimal Colored} to the restriction of $f$ to $\Delta_{n-1}$, we have a proper $r$-Tverberg tuple $(\sigma_1,\ldots, \sigma_r)$ for $f\mid_{\Delta_{n-1}}$ such that the vertex set $I_j\subset [n]$ of each $\sigma_j$ contains at most one element of each color class.

In the real case, let $H_j$, $B_j$, $H''_j$, and $B''_j$ be as in the proof of Theorem~\ref{thm:Dolnikov}(\ref{thm:real Dolnikov}), and likewise in the complex case let $F_r$ and $F_r''$ be as in the proof of Theorem~\ref{thm:Dolnikov}(\ref{thm:complex Dolnikov}).  Thus $F''_r=\cup_{j=1}^r B''_j$ is an $r$-fan in $\K^{n-d-1}$ (conical when $\K=\R$ and complex regular when $\K=\C$) which distributes $X$ such that $x_i$ lies in $\Int(B''_j)$ for all $i\in I_j$ and all $j\in [r]$, while $x_i$ lies in the center of $F_r''$ for all $i\in [n]\setminus \sqcup_{j=1}^r I_j$. To complete the proof, we show that $|\Int(B_j'')\cap X_k|\leq 1$ for all $j$ and $k$. But $|\Int(B_j'') \cap X_k|=|I_j\cap C_k|\leq 1$ for all $j$ and $k$, so the proof is complete.\end{proof}

\section{A Sarkaria-type Theorem for Two Tverberg Tuples}
\label{sec:CS-TM}

In this section we prove Theorem~\ref{thm:Sarkaria 2}, our extension of  Sarkaria's theorem to two Tverberg $r$-tuples. Our proof follows the Configuration Space--Test Map paradigm typically used in topological combinatorics (see, e.g., ~\cite{Zi17}), and in particular closely matches that of an extension of the $r=2$ case of Sarkaria's theorem given in~\cite{FS24}. In our case, the existence of the desired pair of Tverberg tuples for the given map $f\colon \Delta_{n-1}\rightarrow \R^d$ will follow from the existence of a zero of a naturally associated equivariant map from a certain topological space equipped with a free $(\Z_r\oplus \Z_r)$-action to a specific $(\Z_r\oplus \Z_r)$-module. 

\subsection{Configuration Space} Our configuration space is based on the ``deleted product'' scheme commonly used in topological Tverberg-type theory (see, e.g.,~\cite{Oz87, MW15, Vo96}). Given a simplex $\Delta_{n-1}$, recall that the support $\supp(x)$ of a point $x$ in $\Delta_{n-1}$ is the smallest (closed) face in $\Delta_{n-1}$ containing $x$. For $r\geq 2$, the deleted $r$-fold product \[(\Delta_{n-1})^{\times r}_\Delta=\{x=(x_0,\ldots, x_{r-1})\mid x_i\in \Delta_{n-1}\,\, \text{for all}\,\, i \,\,\text{and}\,\, \supp(x_i)\cap \supp (x_j)=\emptyset\,\,\text{for all}\,\, i\neq j\}\] 
is the subset of the Cartesian product $\Delta_{n-1}^{\times r}$ consisting of all $r$-tuples of elements of $\Delta_{n-1}$ whose supports are pairwise disjoint. There is a natural free $\Z_r$-action on $(\Delta_{n-1})^{\times r}_\Delta$, so that for each $i\in \Z_r$ one  cyclically shifts each coordinate of an $r$-tuple $x=(x_0,\ldots, x_{r-1})$ forward by $i$. We shall be concerned with the Cartesian product \[X(r,n)=(\Delta_{n-1})^{\times r}_\Delta \times (\Delta_{n-1})^{\times r}_\Delta\] of the deleted product with itself, which we equip the standard $(\Z_r\oplus \Z_r)$ action. Thus each $\Z_r$-factor of $\Z_r\oplus \Z_r$ acts independently on the corresponding $(\Delta_{n-1})^{\times r}_\Delta$-factor of $X(r,n)$. 

\subsection{Test Space} Our test space will be the direct sum of two $(\Z_r\oplus \Z_r)$-modules. First, consider the regular representation $\R[\Z_r\oplus \Z_r]=\{\sum_{(i,j)\in \Z_r\oplus \Z_r} r_{(i,j)}\cdot(i,j) \mid r_{(i,j)}\in \R\,\,\text{for all}\,\, (i,j)\in \Z_r\oplus \Z_r\}$, for which the  $(\Z_r\oplus \Z_r)$-action is given by setting $(k,\ell)\cdot (i,j)=(k+i,\ell+j)$ for each $(k,\ell),(i,j)\in \Z_r\oplus \Z_r$ and extending linearly. Inside the regular representation is the subrepresentation $U_r=\{\sum_{(i,j)\in \Z_r\oplus \Z_r}r_{(i,j)}\cdot (i,j)\mid \sum_{(i,j)} r_{i,j}=0\}$ consisting of all elements of the regular representation whose coefficients sum to zero. In other words, $U_r$ is the complement of the diagonal 1-dimensional trivial subrepresentation in $\R[\Z_r\oplus \Z_r]$ and so has dimension $r^2-1$. For each $m\geq 1$, we then let $U_r^{\oplus m}$ be the $m$-fold direct sum of $U_r$, which we equip with the diagonal $(\Z_r\oplus \Z_r)$-action. 

For the second representation, we consider $\R^d[\Z_r]=\{\sum_{i\in \Z_r} y_i\cdot i\mid y_i\in \R^d\,\,\text{for all}\,\, i\in \Z_r\}$, which can be identified with the $d$-fold sum of the regular representation of $\Z_r$. As before, this is defined by setting $i\cdot j =i+j$ for each $i,j\in \Z_r$ and extending linearly, which results in a forward cyclic shift of the coefficients. We now let $V_d=\{\sum_{i\in \Z_r} y_i\cdot i\mid \sum_i y_i=0\}$ be the orthogonal complement of the diagonal $d$-dimensional trivial  subrepresentation inside $\R^d[\Z_r]$. Letting each $\Z_r$-factor of $\Z_r\oplus \Z_r$ act independently on each $V_d$-factor of $V_d\oplus V_d$  defines a $(\Z_r\oplus \Z_r)$-subrepresentation of $\R^d[\Z_r]\oplus \R^d[\Z_r]$ of dimension $2d(r-1)$. 

Finally, we let 
\begin{equation}
\label{eqn:U}
U:=(V_d\oplus V_d)\oplus U_r^{\oplus m}.\end{equation}

\subsection{Test Map}
Let $d\geq 1$ be an integer, let $r\geq 3$ be an odd integer, and suppose that $n\geq (r-1)(d+1)+m\frac{r^2-1}{2}+1$. We assume that $f\colon \Delta_{n-1}\rightarrow \R^d$ is a continuous map and that $\F$ is a family of non-empty subsets of $[n]=\Ver(\Delta_{n-1})$ such that $\chi(\KG^{r^2}\F)\leq m$. Thus $\F=\F_1\cup\cdots \cup \F_m$ where none of the $m$ subfamilies $\F_k$ contains $r^2$ pairwise disjoint sets. Without loss of generality, we may assume that each $\F_k$ is closed under taking supersets. Thus we wish to find two Tverberg $r$-tuples $\sigma=(\sigma_1, \ldots,\sigma_r)$ and $\tau=(\tau_1, \ldots, \tau_r)$ such that no $\Ver(\sigma_i\cap\tau_j)$ is a member of $\F$.

First, we construct a continuous map $T\colon X(r,n)\rightarrow V_d\oplus V_d$ whose zeros correspond to pairs of Tverberg tuples. For each $x=(x_0,\ldots, x_{r-1}) \in (\Delta_{n-1})^{\times r}_\Delta$, we let $\Avg(f,x)=\frac{1}{r}\sum_{i=0}^{r-1} f(x_i)$ denote the average value of the $f(x_i)$, and for each $(x,y)\in X(r,n)$ we set \[T(x,y)=(\sum_{i\in \Z_r}(f(x_i)-\Avg(f,x))\cdot i)\oplus (\sum_{i\in \Z_r}(f(y_i)-\Avg(f,y)\cdot i)).\] Clearly, both $(\supp(x_0),\ldots, \supp(x_{r-1}))$ and $(\supp(y_0),\ldots, \supp(y_{r-1}))$ are Tverberg tuples for $f$ if and only if $T(x,y)=0$. It is easily confirmed that $T$ is $(\Z_r\oplus \Z_r)$-equivariant with respect to the actions on the domain and codomain. 

Next, we construct a continuous $(\Z_r\oplus \Z_r)$-equivariant map $D\colon X(r,n)\rightarrow U_r^{\oplus m}$ which ensures, for each zero $(x,y)$ of $T$, that $\Ver(\supp(x_i)\cap \supp(y_j))$ is not an element of $\F$ for any $(i,j)\in \Z_r\oplus \Z_r$. Let $k\in[m]$. As $\F_k$ is closed under taking supersets, $\Sigma_k:=\{\sigma\subseteq [n]\mid \sigma\notin \F_k\}$ is an abstract simplicial complex on $[n]$. We denote by $\|\Sigma_k\|$ the realization of $\Sigma_k$ as a geometric subcomplex of $\Delta_{n-1}$.

Viewing $\Delta_{n-1}$ as a subset of $\R^n$, we endow $\Delta_{n-1}\times \Delta_{n-1}\subset \R^n\times \R^n=\R^{2n}$ with the Euclidian metric $d\colon (\Delta_{n-1}\times \Delta_{n-1})^2\rightarrow [0,\infty)$. For each $k\in [m]$ we define the map $d_k\colon X(r,n)\rightarrow U_m$ by setting \[d_k(x,y)=\sum_{(i,j)\in \Z_r\oplus \Z_r} \left[d((x_i,y_j),\|\Sigma_k\|\times |\Sigma_k\|)-a(x,y)\right]\cdot (i,j),\] where $a(x,y)=\frac{1}{r^2}\sum_{(i,j)\in \Z_r\oplus \Z_r} d\left((x_i,y_j),\|\Sigma_k\|\times |\Sigma_k\|\right)$ is the average of the distances from $(x_i,y_j)$ to $\|\Sigma_k\|\times \|\Sigma_k\|$. Again, it is easily verified that each $d_k$ is equivariant with respect to the actions concerned. 

We now show that $d_k(x,y)=0$ for some $(x,y)\in X(r,n)$ implies that no $\Ver(\supp(x_i) \cap (\supp(y_j))$ is an element of $\F_k$. Supposing that $d_k(x,y)=0$, we have that the distance from each $(x_i,y_j)$ to $\|\Sigma_k\|\times \|\Sigma_k\|$ is equalized. Since the $\supp(x_i)\cap \supp(y_i)$ are pairwise disjoint and $\F_k$ does not contain $r^2$ pairwise disjoint sets, there must be some $(i_0,j_0)\in \Z_r\oplus \Z_r$ such that $\Ver(\supp(x_{i_0})\cap (\supp(y_{j_0}))\in \Sigma_k$. Thus $d((x_{i_0},y_{j_0}), \|\Sigma_k\|\times \|\Sigma_k\|)=0$ and so $d((x_i,y_j),\|\Sigma_k\|\times \|\Sigma_k\|)=0$ for all $(i,j)\in \Z_r\oplus \Z_r$. For each $(i,j)\in \Z_r\oplus \Z_r$ we therefore have that both $x_i$ and $y_j$ lie in $\|\Sigma_k\|$, so $\supp(x_i)\cap \supp(y_j)$ is a face of $\|\Sigma_k\|$ and $\Ver(\supp(x_i) \cap (\supp(y_j))$ is an element of $\Sigma_k$. Thus $\Ver(\sigma_i \cap \tau_j)$ is not a member of $\F_k$ and so cannot contain any member of $\F_k$.

Finally, we define \[D=(d_1,\ldots, d_m)\colon X(r,n)\rightarrow U_r^{\oplus m}\] 
and let \begin{equation}\label{eqn:F} F=T\oplus D\colon X(r,n)\rightarrow U,\end{equation} which is again $(\Z_r\oplus \Z_r)$-equivariant. Given the discussion above, we have that $F(x,y)=0$ implies that $(\supp(x_0),\ldots, \supp(x_{r-1}))$ and $(\supp(y_0),\ldots, \supp(y_{r-1}))$ are Tverberg tuples for $f$ such that none of the $\Ver(\supp(x_i)\cap \supp(y_j))$ contains an element of $\F$. Theorem~\ref{thm:Sarkaria 2} is then an immediate consequence of the following proposition, whose proof we defer to the next section. 

\begin{proposition}
\label{prop:Borsuk-Ulam}
Let $d\geq 1$ be an integer, let $r\geq 3$ be an odd prime, let $n\geq (r-1)(d+1)+m\frac{r^2-1}{2}$ be an integer, and let $U$ be the representation~(\ref{eqn:U}). If the $r$-ary expansion of $m(r-1)/2$ has all even coefficients, then any continuous $(\Z_r\oplus \Z_r)$-equivariant map $F\colon X(r,n)\rightarrow U$ has a zero.  
\end{proposition}

\begin{proof}[Proof of Theorem~\ref{thm:Sarkaria 2}] Let $f\colon \Delta_{n-1}\rightarrow \R^d$ be continuous, and suppose that $\F$ is a family of non-empty subsets of $[n]$ for which $\chi(\KG^{r^2}(\F))\leq m$. As the  map $F\colon X(r,n)\rightarrow U$ given by (\ref{eqn:F}) is equivariant, by Proposition~\ref{prop:Borsuk-Ulam} there is some $(x,y)\in X(r,n)$ such that $F(x,y)=0$. Hence $(\supp(x_0),\ldots, \supp(x_{r-1}))$ and $(\supp(y_0),\ldots, \supp(y_{r-1}))$ are the desired Tverberg tuples for $f$.\end{proof}

\begin{remark}
\label{rem:values}
We conclude this section by verifying that Proposition~\ref{prop:Borsuk-Ulam} (and Theorems~\ref{thm:Sarkaria 2}, ~\ref{thm:equidistribute 2}, and ~\ref{thm:Dolnikov 2}) holds for the particular values of $m$ mentioned in the introduction, namely when \begin{itemize} \item $m=2a(r^{\ell_1}+\cdots+r^{\ell_k})$ where $0\leq \ell_1<\cdots<\ell_k$ and $1\leq a\leq r-1$ is odd, or \item $r\equiv 1\pmod{4}$ and $m=r^{\ell_1}+\cdots+r^{\ell_k}$ with $0\leq \ell_1<\cdots<\ell_k$. \end{itemize} 

In the first case we have have \begin{align*}
    m(r-1)/2 & =a(r-1)(r^{\ell_1}+\cdots+r^{\ell_k})\\
    & =(r-a)r^{\ell_1}+(a-1)r^{\ell_1+1}+(r-a)r^{\ell_2}+(a-1)r^{\ell_2+1}+\cdots+(r-a)r^{\ell_k}+(a-1)r^{\ell_k+1}.
\end{align*}
    
As both $a-1$ and $r-a$ are even, we are done provided that $\ell_i+1<\ell_{i+1}$ for all $i\in [k]$. On the other hand, if $\ell_i+1=\ell_{i+1}$ for some $i$, then the coefficient of $r^{\ell_{i+1}}$ in the $r$-ary expansion of $m(r-1)/2$ is $r-1$, which is again even. Thus each coefficient in the $r$-ary expansion of $m(r-1)/2$ is even in all cases. As $1+r+\cdots+r^\ell=\frac{r^{\ell+1}-1}{r-1}$, Proposition~\ref{prop:Borsuk-Ulam} applies in particular if $m=2a(r^{\ell+1}-1)/(r-1)$ and $\ell\geq 0$. In the second case, $(r-1)/2$ is even and so $m(r-1)/2=\sum_{i=1}^k \frac{r-1}{2}\,r^{\ell_i}$. Thus the coefficients of the $r$-ary expansion are all even. In particular, one may let $m=\frac{r^{\ell+1}-1}{r-1}$ for all $\ell\geq 0$.\end{remark}

\section{Proof of Proposition~\ref{prop:Borsuk-Ulam}} 
\label{sec:BU Proof}
Our proof of Proposition~\ref{prop:Borsuk-Ulam} follows a standard factorization trick, whereby the given equivariant map induces a section of a complex vector bundle given by the Borel construction. One then shows that the bundle's top Chern class is non-zero, so that the section, and therefore the equivariant map, necessarily has a zero. A similar Chern class computation can be found in ~\cite{Si15} in the context of mass equipartitions by complex regular fans. We refer the reader to ~\cite{Hu94, Mil74} for the basic theory of vector bundles and characteristic classes and to ~\cite{FH04, Ser77, St12} for more on the representation theory of finite groups.  

\subsection{Complex Representations} We begin by showing that the real $(\Z_r\oplus \Z_r)$-module $U$ (\ref{eqn:U}) above can be equivariantly identified with the complex $(\Z_r\oplus \Z_r)$-module $W$  (\ref{eqn:W}) below.  

To that end, recall that the irreducible complex representations of any finite abelian group $G$ are all 1-dimensional (see, e.g., ~\cite[\S 3.1 Theorem 9]{Ser77}) and are indexed by the elements of the group itself (see, e.g.,~\cite[\S 2.7 Theorem 7]{Ser77}). Explicitly, if $G=\Z_{m_1}\oplus\cdots\oplus \Z_{m_k}$, then for each $a=(a_1,\ldots, a_k)\in G$ the homomorphism  $\rho_a\colon G\rightarrow \C^\ast$ is given by $\rho_a(g)=\omega_{m_1}^{a_1g_1}\cdots \omega_{m_k}^{a_kg_k}$ for each $g=(g_1,\ldots, g_k)\in G$, where $\omega_{m_j}=\exp(2\pi i/m_j)$ is the standard $m_j$-th root of unity (see e.g.,~\cite[Proposition 4.5.1]{St12}). We denote the corresponding $G$-module by $W_a=\C$. Letting $\C[G]=\{\sum_{g\in G} z_g\cdot g\mid z_g\in G\,\,\text{for all}\,\, g\in G\}$ denote the complex regular representation, one has moreover that $\C[G]\cong \oplus_{a\in G}W_a$ (see, e.g.,~\cite[\S 2.4 Corollary 1]{Ser77}).

Next, recall that if $G$ is any finite group and $V$ is a $d$-dimensional real $G$-representation then its complexification $V\otimes \C$ (see, e.g.,~\cite{FH04}) is a $d$-dimensional complex $G$-representation. For our purposes, we have that the complexification of the real regular representation $\R[\Z_r\oplus \Z_r]$ is the complex regular representation $\C[\Z_r\oplus \Z_r]$, and likewise the complexifications of $V_d\oplus V_d$ and $U_r$ of Section~\ref{sec:CS-TM} are obtained by replacing real with complex coefficients in their respective definitions. As $\C[\Z_r\oplus \Z_r]\cong \oplus_{(a,b)\in \Z_r\oplus \Z_r}W_{(a,b)}$, we have that $U_r\otimes \C\cong\oplus_{(a,b)\neq (0,0)} W_{(a,b)}$ and so that $U_r^{\oplus m}\otimes \C\cong \oplus_{(a,b)\neq 0}W_{(a,b)}^{\oplus m}$. Now consider the $(\Z_r\oplus \Z_r)$-representation $V_d\oplus V_d$. As $\R[\Z_r]\otimes \C\cong\C[\Z_r]\cong \oplus_{a\in \Z_r}W_a$, we have that $V_d\otimes \C\cong \oplus_{a\neq 0} W_a^{\oplus d}$ as a $\Z_r$-representation. However, projecting $\Z_r\oplus \Z_r$ onto the first factor allows one to view the first $V_d$-factor of $V_d\oplus V_d$ as a $(\Z_r\oplus \Z_r)$-module, from which one obtains the decomposition $V_d\otimes \C\cong \oplus_{a\neq 0} W_{(a,0)}^{\oplus d}$. Proceeding analogously for the second $V_d$-factor, we conclude that $(V_d\oplus V_d)\otimes \C\cong\bigoplus_{a\neq 0} W_{(a,0)}^{\oplus d}\oplus \bigoplus_{b\neq 0} W_{(0,b)}^{\oplus d}$. Thus
\[U\otimes \C \cong \bigoplus_{a\neq 0} W_{(a,0)}^{\oplus d}\oplus \bigoplus_{b\neq 0} W_{(0,b)}^{\oplus d}\oplus \bigoplus_{(a,b)\neq 0}W_{(a,b)}^{\oplus m}.\]

On the other hand, if $G$ is a finite group and $V$ is a $d$-dimensional complex $G$-representation then one may consider the $2d$-dimensional real $G$-representation $V_\R$ given by restricting scalar multiplication in $V$ to $\R$ (see, e.g.,~\cite{FH04}). If $U$ is a real representation, then any element of its complexification $U\otimes \C$ can be expressed uniquely in the form $u_1+iu_2:=u_1\otimes 1 + u_2\otimes i$ with $u_1,u_2\in U$, and the $G$-action on $U\otimes \C$ is given by $g\cdot (u_1+iu_2)=g\cdot u_1+i(g\cdot u_2)$. One therefore has the isomorphism $(U\otimes \C)_\R\cong U\oplus U$ of real representations. Now let $(a,b)\in \Z_r\oplus \Z_r$ and consider the $1$-dimensional complex $(\Z_r\oplus\Z_r)$-module $W_{(a,b)}$ above. We have that $W_{(-a,-b)}=\overline{W_{(a,b)}}$ is the complex conjugate representation of $W_{(a,b)}$, and since complex conjugation is a real linear isomorphism of $\R^2$ we conclude that $(W_{(a,b)})_\R$ and $(W_{(-a,-b)})_\R$ are isomorphic as real representations. Finally, let $A$ be the subset of $\Z_r\oplus \Z_r$ consisting of all non-zero $(a,b)$ whose last non-zero coordinate lies in $\{1,\ldots, \frac{r-1}{2}\}$. We then have that
\[U\oplus U\cong \bigoplus_{1\leq a\leq \frac{r-1}{2}}(W_{(a,0)})_\R^{\oplus 2d}\oplus \bigoplus_{1\leq b\leq \frac{r-1}{2}}(W_{(0,b)})_\R^{\oplus 2d}\oplus \bigoplus_{(a,b)\in A}(W_{(a,b)})_\R^{\oplus 2m}\] and therefore that $U\cong W_\R$ where $W$ is the complex representation \begin{equation} 
\label{eqn:W} 
W:= \bigoplus_{1\leq a \leq \frac{r-1}{2}}W_{(a,0)}^{\oplus d}\oplus \bigoplus_{1\leq b \leq \frac{r-1}{2}} W_{(0,b)}^{\oplus d}\oplus \bigoplus_{(a,b)\in A}W_{(a,b)}^{\oplus m}.\end{equation} 
We note that the complex dimension of $W$ is $\dim W=N:=(r-1)d+m\frac{r^2-1}{2}.$

\subsection{A Chern Class Computation} Let $n>r$ be an integer. For simplicity, we will let $X_n=(\Delta_{n-1})^{\times r}_\Delta$ denote the deleted $r$-fold product of Section~\ref{sec:CS-TM} and denote by $\overline{X_n}=X_n/\Z_r$ its quotient by the free $\Z_r$-action discussed there. The quotient of $X(r,n)=X_n\times X_n$ under the standard $(\Z_r\oplus \Z_r)$-action is then $\overline{X(r,n)}=\overline{X_n}\times \overline{X_n}$.

For our purposes, the central feature of $X_n$ is that it is an $(n-r)$-dimensional cellular complex which is $(n-r-1)$-connected (see, e.g,~\cite{BBS81}). The union $\bigcup_{n=r}^\infty X_n$ may therefore be taken to be the total space $E\Z_r$ of the classifying bundle $\Z_r\hookrightarrow E\Z_r\rightarrow B\Z_r:=E\Z_r/\Z_r$ for $\Z_r$, with classifying space $B\Z_r$  taken to be $\bigcup_{n=r}^\infty\overline{X_n}$. One may then take $B\Z_r\times B\Z_r$ to be the classifying space $B(\Z_r\oplus \Z_r)$ for $\Z_r\oplus \Z_r$, of which $\overline{X(r,n)}$ is a subcomplex.

Now let $n\geq (r-1)(d+1)+m\frac{r^2-1}{2}+1$ and suppose that $F\colon X(r,n)\rightarrow W$ is a $(\Z_r\oplus \Z_r)$-equivariant map. As the inclusion $X(r,n_1) \hookrightarrow X(r,n_2)$ is $(\Z_r \oplus \Z_r)$-equivariant when $n_2>n_1>r$ we may assume that $n=(r-1)(d+1)+m\frac{r^2-1}{2}+1$. Letting $N=(r-1)d+m(r^2-1)/2$ as above gives that $N=n-r$ and that $\dim X(r,n)=2N$.

Corresponding to the $N$-dimensional complex representation $W$, let \[E(r,n):=X(r,n)\times_{\Z_r\oplus \Z_r} W\] be the $N$-dimensional complex vector bundle over $\overline{X(r,n)}$ obtained by  quotienting the trivial bundle $X(r,n)\times W$ under the diagonal $(\Z_r\oplus \Z_r)$-action. The map $F$ gives rise to a section $s\colon \overline{X(r,n)}\rightarrow E(r,n)$ of this bundle which is induced by the equivariant section $X(r,n)\rightarrow X(r,n) \times W$ of the trivial bundle given by $x\mapsto (x,F(x))$. Thus $F$ has a zero if and only if the section $s$ does. However, a never-vanishing section of $E(r,n)$ exists if and only if the bundle's top Chern class $c_N(E(r,n))\in H^{2N}(\overline{X(r,n)};\Z)$ is zero. To complete the proof we will show that the mod $r$-reduction $c_N(E(r,n);\Z_r)\in H^{2N}(\overline{X(r,n)};\Z_r)$ of $c_N(E(r,n))$ is non-zero.

Let $Y_n=E\Z_r^{(N)}$ denote the $N$-dimensional skeleton of $E\Z_r$, so that $\overline{Y_n}=Y_n/\Z_r$ is the $N$-skeleton of $B\Z_r$. Likewise, let $Y(r,n)=Y_n\times Y_n$ and $\overline{Y(r,n)}=Y(r,n)/(\Z_r\oplus \Z_r)=\overline{Y_n}\times \overline{Y_n}$. Since $X_n$ is $(N-1)$-connected and the free action of $\Z_r$ on $X_n$ is cellular, its identity map extends to a $\Z_r$-equivariant map $Y_n\rightarrow X_n$.  The identify map of $X(r,n)$ therefore extends to a $(\Z_r\oplus \Z_r)$-equivariant map $ Y(r,n)\to X(r,n)$ with a resulting map $h\colon \overline{Y(r,n)}\rightarrow \overline{X(r,n)}$ on quotients. Letting \[E'(r,n):=Y(r,n)\times_{\Z_r\oplus \Z_r} W,\] the map $Y(n,r)\rightarrow X(n,r)$ induces a bundle map $ E'(r,n)\rightarrow E(r,n)$ and therefore $c_N(E'(r,n);\Z_r)=h^*(c_N(E(r,n);\Z_r))$ by the naturality of Chern classes. Thus to show that $c_N(E(r,n);\Z_r)\neq 0$ it suffices to show that $c_N(E'(r,n);\Z_r)\in H^{2N}(\overline{Y(r,n)};\Z_r)\neq 0$. 
    
Next, recall that the total Chern class $c(V)$ (or its mod $r$ reduction $c(V;\Z_r)$) of any finite-dimensional complex representation $V$ of a finite group $G$ is defined to be the (mod $r$ reduction of the) total Chern class of the vector bundle $E_V:=EG\times_G V$ over its classifying space $BG$ (see, e.g.,~\cite{At61}). In our case, the bundle $E'(r,n)$ above is the pullback under inclusion $i\colon \overline{Y(r,n)}\hookrightarrow B(\Z_r\oplus \Z_r)$ of the bundle $E_W$ determined by the representation $W$~(\ref{eqn:W}). Thus $c_N(E'(r,n);\Z_r)=i^\ast(c_N(W);\Z_r)$ by the naturality of Chern classes.

We now calculate $c_N(W;\Z_r)\in H^\ast(B(\Z_r\oplus \Z_r);\Z_r)$. Recall that $H^\ast (B\Z_r;\Z_r)\cong \Lambda[x]\otimes_{\Z_r} \Z_r[y],$ where $|x|=1$, $|y|=2$, and $\Lambda[\cdot ]$ denotes the exterior algebra. Moreover, one may take $y=c_1(W_1;\Z_r)$ to be the mod $r$ reduction of the standard complex 1-dimensional representation $W_1$ of $\Z_r$ (see, e.g., \cite{At61} or ~\cite{Ha00}). Letting $\pi_j\colon B\Z_r\times B\Z_r\rightarrow B\Z_r$ be the $j$-th coordinate projection for $j\in\{1,2\}$, the K\"unneth formula (see, e.g.,~\cite{Ha00}) gives that \[H^\ast(B(\Z_r\oplus \Z_r);\Z_r)\cong \Lambda[x_1,x_2]\otimes_{\Z_r}\Z_r[y_1,y_2],\] where $x_j=\pi_j^*(x)$ and $y_j=\pi_j^*(y)$ for $j\in\{1,2\}$. However, $E_{W_{(1,0)}}=\pi_1^\ast(E_{W_1})$ and $E_{W_{(0,1)}}=\pi_2^\ast (E_{W_1})$ are the pullbacks of $E_{W_1}$ under the respective coordinate projections, so by the naturality of Chern classes we have $y_1=c_1(W_{(1,0)};\Z_r)$ and $y_2=c_2(W_{0,1};\Z_r)$. Next, the Whitney sum formula (see, e.g., ~\cite{Mil74}) shows that $c(V_1\oplus V_2;\Z_r)=c(V_1;\Z_r)c(V_2;\Z_r)$ for any two complex $(\Z_r\oplus \Z_r)$-representations $V_1$ and $V_2$, and one also has $c_1(V_1\otimes V_2;\Z_r)=c_1(V_1;\Z_r)+c_1(V_2;\Z_r)$ if both of these representations are $1$-dimensional (see, e.g, ~\cite{Hu94}). As $W_{(j,k)}=W_{(1,0)}^{\otimes j}\otimes W_{(0,1)}^{\otimes k}$ for any $(j,k)\in \Z_r\oplus \Z_r$ and $c(E;\Z_r)=1+c_1(E;\Z_r)$ for any line bundle $E$, the decomposition (\ref{eqn:W}) of $W$ gives 
\[c_N(W;\Z_r)=[((r-1)/2)!]^{2d}y_1^{d(r-1)/2}y_2^{d(r-1)/2}\prod_{(a_1,a_2)\in A}(a_1y_1+a_2y_2)^m,\] where we recall that $A\subset \Z_r\oplus \Z_r$ consists of all non-zero $(a_1,a_2)\in \Z_r\oplus \Z_r$ whose last non-zero coordinate lies in $\{1,\ldots, (r-1)/2\}$. 

For each $k\in \{1,\ldots, (r-1)/2\}$, let $A_k$ be the subset of $A$ consisting of all $(a_1,a_2)$ whose last non-zero coordinate is equal to $k$. Let $k^{-1}\in \Z_r$ denote the multiplicative inverse of $k$ in $\Z_r$. Multiplication by a fixed an element of $\Z_r\setminus\{0\}$ determines a bijection of $\Z_r\setminus\{0\}$, and therefore \begin{align*} \prod_{(a_1,a_2)\in A_k}(a_1y_1+a_2y_2)&=ky_1\cdot ky_2\cdot (y_1+ky_2)\cdot (2y_1+ky_2)\cdots((r-1)y_1+ky_2)\\
& = k^{r+1}y_1\cdot y_2\cdot (k^{-1}y_1+y_2)\cdot(2k^{-1}y_1+y_2)\cdots((r-1)k^{-1}y_1+y_2)\\
& = k^{r+1}y_1\cdot y_2\cdot (y_1+y_2)\cdot (2y_1+y_2)\cdots ((r-1)y_1+y_2)\\
&=k^{r+1}\prod_{(a_1,a_2)\in A_1}(a_1y_1+a_2y_2).\end{align*}
Thus $\prod_{(a_1,a_2)\in A}(a_1y_1+a_2y_2)=[((r-1)/2)!]^{r+1}\prod_{(a_1,a_2)\in A_1}(a_1y_1+a_2y_2)^{(r-1)/2}.$
On the other hand, $\prod_{(a_1,a_2)\in A_1}(a_1y_1+a_2y_2)=y_1y_2^r-y_2y_1^r$ (see, e.g.,~\cite[Proof of Proposition 1.1]{Wi83}), so $c_N(W;\Z_r)=cp(y_1,y_2)$ where $c$ is a non-zero constant and \[p(y_1,y_2)=y_1^{d(r-1)/2}y_2^{d(r-1)/2}(y_1y_2^r-y_2y_1^r)^{m(r-1)/2}.\] 

As $\overline{Y_n}$ is the $N$-dimensional skeleton of $B\Z_r$, the long exact sequence for the pair $(B\Z_r, \overline{Y_n})$ shows that the map $i^\ast \colon H^k(B\Z_r;\Z_r)\cong\Z_r\rightarrow H^k(\overline{Y_n};\Z_r)$ induced from the inclusion $i\colon \overline{Y_n}\hookrightarrow B\Z_r$ is injective whenever $k\leq N$, while on the other hand it is the zero map for all $k>N$. Now suppose $k_1,k_2\geq 0$ with $k_1+k_2=2N$ and consider the resulting map \[H^{k_1}(B\Z_r;\Z_r)\otimes_{\Z_r} H^{k_2}(B\Z_r;\Z_r)\longrightarrow  H^{k_1}(\overline{Y_n};\Z_r)\otimes_{\Z_r} H^{k_2}(\overline{Y_n};\Z_r)\] on tensor products. This is the zero map unless $k_1=k_2=N$. Because we are in the category of modules over the field $\Z_r$ (i.e., $\Z_r$-vector spaces), this map is injective when $k_1=k_2=N$. Now let $N_0=N/2=d\frac{r-1}{2}+m\frac{r^2-1}{4}$. As $H^*(\overline{Y(r,n)};\Z_r)\cong H^\ast (\overline{Y_n};\Z_r)\otimes_{\Z_r} H^\ast(\overline{Y_n};\Z_r)$ by the K\"unneth formula, we conclude from the functoriality of the K\"unneth formula that $c_N(E'(r,n);\Z_r)=i^*(c_N(W;\Z_r))\in H^{2N}(\overline{Y(r,n)};\Z_r)$ is non-zero if and only if the coefficient of $y_1^{N_0}y_2^{N_0}$ in $p(y_1,y_2)$ is non-zero.

Now consider the $r$-ary expansion $b_0+b_1r+\cdots+b_\ell r^\ell$ of $m':=m(r-1)/2$. These coefficients are all even by assumption, so $m'$ is even as well. Now observe that the coefficient of $y_1^{N_0}y_2^{N_0}$ in $p(y_1,y_2)$ is $\binom{m'}{m'/2}$. However, $\binom{m'}{m'/2}\equiv\prod_{i=0}^\ell \binom{b_i}{b_i/2} \pmod{r}$ by Lucas's theorem (see, e.g., ~\cite{Fi47}) and so $\binom{m'}{m'/2}$ is non-zero in $\Z_r$. We conclude therefore that $c_N(E'(r,n);\Z_r)\neq 0$ and hence that $c_N(E(r,n);\Z_r)\neq 0$ as well. This completes the proof of Proposition~\ref{prop:Borsuk-Ulam}. 

\section{Piercing Distributions for Two Fans}
\label{sec:Two Fans}
In this section we prove our equidistribution theorem for two fans. As was the case with a single fan, this follows from a corresponding Dolnikov-type statement. 

\begin{theorem} 
\label{thm:Dolnikov 2} 
Let $d\geq 1$ be an integer, let $r\geq 3$ be a prime number, and let $m\geq 1$ be an integer such that all of the coefficients in the $r$-ary expansion of $m(r-1)/2$ are even. 

\begin{compactenum}[(a)]

\item \label{thm:real Dolnikov 2}
Suppose that $X$ is a set of $n\geq (r-1)(d+1)+\frac{m(r^2-1)}{2}+1$ affinely spanning points in $\R^{n-d-1}$. If $\F$ is a family of subsets of $X$ with $\chi(\KG^{r^2}(\F))\leq m$, then there exists a pair of conical $r$-fans $F_r^1$ and $F_r^2$ whose intersection contains $X$ such that every $A\in \F$ is intersected by two closed half-flats from some $F_r^i$. 

\item \label{thm:complex Dolnikov 2}
Suppose that $X$ is a set of $n\geq (r-1)(2d+1)+\frac{m(r^2-1)}{2}+1$ complex affinely spanning points in $\C^{n-d-1}$. If $\F$ is a family of subsets of $X$ such that $\chi(\KG^{r^2}(\F))\leq m$, then there exists a pair of complex regular $r$-fans $F_r^1$ and $F_r^2$ in $\C^{n-d-1}$ whose intersection contains $X$ such that every $A\in \F$ is intersected by two closed half-hyperplanes from some $F_r^i$. 
\end{compactenum}
\end{theorem}

\begin{proof}[Proof of Theorem~\ref{thm:equidistribute 2}]

We prove the real case of Theorem~\ref{thm:equidistribute 2}; the complex case is identical. So let $n\geq (r-1)(d+1)+m\frac{r^2-1}{2}$+1, let $m$ satisfy the condition of Theorem~\ref{thm:equidistribute 2}(\ref{thm:real equidstirubte 2}), and let $X=X_1\sqcup\cdots\sqcup X_m$ be an $m$-coloring of a set $X$ of $n$ affinely spanning points in $\R^{n-d-1}$. Corresponding to the coloring, we let $\F_k=\{A\subseteq X\mid |A\cap X_k|>|X_k|/r^2\}$ for each $k\in [m]$, and we let $\F=\cup_{k=1}^m \F_k$. We have $\chi(\KG^{r^2}(\F))\leq m$ by the pigeon-hole principle, and so by Theorem~\ref{thm:Dolnikov 2} there are two conical $r$-fans $F_r^1=\cup_{i=1}^r B_i^1$ and $F_r^2=\cup_{i=1}^r B_i^2$ in $\R^{n-d-1}$ such no element of $\F$ is completely contained in any of the intersections $\Int(B_1^i)\cap \Int(B_2^j)$. However, $\Int(B_1^i)\cap \Int(B_j^2)\cap X$ is a subset of $\Int(B_1^i)\cap \Int(B_2^j)$ for any $i,j\in [r]$, so none of the intersections $\Int(B_1^i)\cap \Int(B_2^j)$ is a member of $\F$. Thus $|\Int(B_1^i)\cap \Int(B_2^j)\cap X_k|\leq |X_k|/r^2$ for all $i,j\in [r]$ and all $k\in [m]$. \end{proof}

We now prove Theorem~\ref{thm:Dolnikov 2}. As with the case of a single fan, this is done by applying Gale duality to the linear case of Theorem~\ref{thm:Sarkaria 2}. 

\begin{proof}[Proof of Theorem~\ref{thm:Dolnikov 2}] Let $\K=\R$ or $\C$. If $\K=\R$ we let $n\geq (r-1)(d+1)+m\frac{r^2-1}{2}+1$, while if $\K=\C$ we let $n\geq (r-1)(2d+1)+m\frac{r^2-1}{2}+1$. Suppose now that $X=\{x_1,\ldots, x_n\}$ is a set of $n$ affinely spanning points in $\K^{n-d-1}$ and that $\F$ is family of non-empty subsets of $X$ with $\chi(\KG^{r^2}\F)\leq m$. As in the proof of Theorem~\ref{thm:Dolnikov}, we lift $X$ to the $\K$-affine hyperplane $\K^{n-d-1}\times\{1\}$ in $\K^{n-d}$ by setting $g_1=(x_1,1),\ldots, g_n=(x_n,1)$. By including the point $g_{n+1}=-\sum_{i=1}^n g_i$, the sequence $g_1,\ldots, g_{n+1}$ in $\K^{n-d}$ is the Gale transform of a sequence $a_1,\ldots, a_{n+1}$ of $n+1$ affinely spanning points in $\K^d$. As before, we let $f\colon \Delta_n\rightarrow \K^d$ be the unique (real) affine linear map which sends each $i\in [n+1]$ to $a_i\in \K^d$, and, corresponding to $\F$, we let $\F'$ be the family of all subsets $\{i_1,\ldots, i_k\}$ of $[n]$ such that $\{x_{i_1},\ldots, x_{i_k}\}$ is in $\F$. Thus $\chi(\KG^r(\F'))\leq m$. 

As $n$ is of the correct size, we may apply Theorem~\ref{thm:Sarkaria 2} to the restriction $f\mid_{\Delta_{n-1}}\colon \Delta_{n-1}\rightarrow \K^d$ of $f$. This gives two proper Tverberg-tuples $(\sigma_1,\ldots, \sigma_r)$ and $(\tau_1,\ldots, \tau_r)$ for $f\mid_{\Delta_{n-1}}$ such that none of the sets of the form $\Ver(\sigma_i\cap \tau_j)$ contains any element of $\F'$ as a subset. Now let $I^1_i=\Ver(\sigma_i)$ for all $i\in [r]$ and $I_j^2=\Ver(\tau_j)$ for each $j\in [r]$. By Lemma~\ref{lem:linear conical fans} applied to $f$, there exist two linear conical $r$-fans $F_r^1=\cup_{i=1}^r B_i^1$ and $F_r^2=\cup_{j=1}B_j^2$ in $\K^{n-d}$ such that $g_1,\ldots, g_{n+1}$ is distributed according to both the $I^1_i$ and the $I^2_j$. As each $I^1_i$ and each $I^j_2$ is a subset of $[n]$, $g_{n+1}$ lies in the center of both fans.

As in the proof of Theorem~\ref{thm:Dolnikov}, for each $k\in \{1,2\}$ and each $j\in [r]$ we let $(B_j^k)'$ be the intersection of $B_j^k$ with $\K^{n-d-1}\times \{1\}$ and we let $(F_i^k)'=\cup_{j=1}^k (B_j^k)'$. Letting $(B_j^k)''$ be the projection of $(B_j^k)'$ onto $\K^{n-d-1}\times \{0\}$, the identical argument to the one given there shows that each $(F_r^k)''=\cup_{j=1}^k (B_j^k)''$ is an $r$-fan in $\K^{n-d-1}$ (conical when $\K=\R$, and complex regular when $\K=\C$). As $g_1,\ldots, g_{n+1}$ is distributed according to both the $I^1_i$ and the $I^2_j$, we have that $X$ lies in the intersection of the two fans $(F_r^1)''\cap (F_r^2)''$ of these fans and moreover for each $k\in[n]$ that $x_k$ lies in $\Int(B_i^1)\cap \Int(B_j^2)$ if and only if $k\in I^1_j\cap I^2_j$.   We now show that every element of $\F$ is intersected by two half-flats from at least one of the $F_r^k$, or equivalently that no member of $\F$ is contained in any of the intersections of the form $\Int((B_i^1)'')\cap \Int((B_j^2)'')$. So let $i,j\in [r]$, let $A\in \F$, and let $A'$ be the corresponding element of $\F'$.  As $A'\notin I_i^1\cap I_j^2$, there is some $k\in[n]$ with $k\in A'$ and $k\notin I_i^1\cap I_j^2$. Thus $x_k\in A$ but $x_k\notin\Int((B_i^1)'')\cap \Int((B_j^2)'')$.\end{proof}

\section{Robustness for Typical Point Sets and Optimality of Fan Distributions}
\label{sec:optimality and general}

We conclude by showing that there is a dense open class of affinely spanning point sets for which the distributing $r$-fans of Theorems~\ref{thm:Dolnikov} and~\ref{thm:colorful} 
contain a robust number of points from the given point set. Lastly, we prove various cases of the near optimality of Theorem~\ref{thm:Dolnikov} with respect to dimension. 

\subsection{Typical Point Sets}
\label{sec:generic}

Let $a_1,\ldots, a_n$ be a sequence of $n$ ($\K$-) affinely spanning points in $\K^d$ and let $a_{n+1}=\frac{1}{n}\sum_{i=1}^n a_i$ be their average. Letting $A$ be the matrix of the Gale transform associated to $a_1,\ldots, a_{n+1}$, one may choose a basis $b_1,\ldots, b_{n-d}$ for $\ker A$ with $b_{n-d}=(1,\ldots, 1, -n)$, and, for any such choice of such basis, the Gale dual of $a_1,\ldots, a_{n+1}$ is of the  form $g_1=(x_1,1),\ldots, g_n=(x_n,1),g_{n+1}=-\sum_{i=1}^n g_i=-\sum_{i=1}^n (x_i,n)$. We say that the sequence $x_1,\ldots, x_n$ of affinely spanning points in $\K^{n-d-1}$ \emph{corresponds} to $a_1,\ldots, a_n$. Moreover, by letting $g_i=(x_i,1)$ for $1\leq i \leq n$ and $g_{n+1}=-\sum_{i=1}^n g_i$, it is not difficult to see that any sequence $x_1,\ldots, x_n$ of affinely spanning points in $\K^{n-d-1}$ corresponds to some sequence $a_1,\ldots, a_n$ of affinely spanning points in $\K^d$.

Our condition on affinely spanning point sets $X$ with $n$ elements in $\K^{n-d-1}$ is essentially the Gale dual of a generic condition on affinely spanning point sets $A$ of $n$ elements in $\K^d$. We recall that a sequence $a_1,\ldots, a_n$ of $n$ points in $\K^d$ is in \emph{strong general position}~\cite{PS14} if for any $1\leq r\leq n$ and any $r$ pairwise disjoint subsets $I_1,\ldots, I_r$ of $[n]$ one has that the (real) codimension of the intersection of the real affine hulls $\Aff(A_j)$ of the $A_j=\{a_i\mid i\in I_j\}$ equals the sum of their respective codimensions, that is,  \[\codim(\cap_{j=1}^r \Aff(A_j))=\sum_{j=1}^r \codim(\Aff(A_j)).\]  

 The central property of strong general position for our purposes is that it immediately implies that any Tverberg $r$-tuple for a linear map $f\colon \Delta_{n-1}\rightarrow \R^d$ determined by points $a_1,\ldots, a_n$ in strong general position must satisfy $\sum_{j=1}^r |\Ver(\sigma_j)|\geq (r-1)(d+1)+1$. 

 Let $n\geq d+1$. We shall say that an affinely spanning point set $X=\{x_1,\ldots, x_n\}$ in $\K^{n-d-1}$ is \emph{typical} if $x_1,\ldots, x_n$ corresponds to a sequence $a_1,\ldots, a_n$ in $\K^d$ in strong general position. While we expect that typical point sets are generic, we shall content ourselves with showing that they form an open and dense subset of the space of all affinely spanning point sets of size $n$ in $\K^{n-d-1}$. Specifically, let $\Conf_n(\K^{n-d-1})=\{(x_1,\ldots, x_n)\in (\K^{n-d-1})^n\mid x_i\neq x_j \,\,\text{for all} \,\, i\neq j\}$ be the configuration space of sequences of $n$ pairwise distinct points in $\K^{n-d-1}$. If $S$ denotes the subspace of $(\K^d)^n$ consisting of all affinely-spanning sequences of length $n$ and $U$ is the subset of $S$ consisting of all such sequences which are in strong general position, then the generiticity of strong general position in $\K^d$ implies in particular that $U$ is an open dense subset of $S$. An analogous situation holds for typical point sets of size $n$ in $\K^{n-d-1}$. 

\begin{proposition}
\label{prop:generic}
Let $S'$ be the subset $(\K^{n-d-1})^n$ consisting of all affinely spanning sequences of length $n$ in $\K^{n-d-1}$ and let $U'$ be the subset of $S'$ consisting of all such sequences which correspond to a sequence of $n$ points in $\K^d$ in strong general position. Then $U'\cap \Conf_n(\K^{n-d-1})$ is an open dense subset of $S'\cap  \Conf_n(\K^{n-d-1})$. 
\end{proposition}

Proposition~\ref{prop:generic} is an immediate consequence of the following ``continuity'' lemma for the Gale transform.

\begin{lemma}
\label{lem:continuity}

Let $n$ and $d$ be positive integers such that $n\geq d+1$. Suppose that $x=(x_1,\ldots, x_n)$ is a sequence of $n$ affinely spanning points in $\K^{n-d-1}$ which corresponds to a sequence $a=(a_1,\ldots, a_n)$ of $n$ affinely spanning points in $\K^d$.

\begin{compactenum}[(a)]

\item \label{lem:density} Suppose that for each $m\geq 1$ there is a sequence $a_m=(a_1^m,\ldots, a_n^m)$ of $n$ affinely spanning points in $\K^d$ such that $\lim_{m\to\infty} a_m=a$. Then for each $m$ there is a sequence $x_m=(x_1^m,\ldots, x_n^m)$ of $n$ affinely spanning points in $\K^{n-d-1}$ such that each $x_m$ corresponds to $a_m$ and $\lim_{m\to\infty} x_m=x$.  

\item \label{lem:open}
Suppose that for each $m\geq 1$ there is a sequence $x_m=(x_1^m,\ldots, x^m_n)$ of $n$ affinely spanning points in $\K^{n-d-1}$ such that $\lim_{m\to \infty} x_m=x$. Then for each $m$ there is a sequence $a_m=(a_1^m,\ldots, a_m^n)$ of $n$ affinely spanning point sets in $\K^d$ such that each $x_m$ corresponds to $a_m$ and $\lim_{m\to\infty}a_m=a$. 

\end{compactenum}
\end{lemma}

\begin{proof}[Proof of Proposition~\ref{prop:generic}]

As $\Conf_n(\K^{n-d-1})$ is open in $(\K^{n-d-1})^n$ it suffices to prove that $U'$ is an open dense subset of $S'$. 

We first prove density. Suppose that $x=(x_1,\ldots, x_n)$ is a sequence of affinely spanning points in $\K^{n-d-1}$, so that $x$ corresponds to some sequence $a=(a_1,\ldots, a_n)$ of affinely spanning points in $\K^d$. By the density of $U$ in $S$, there is a sequence $(a_m)_{m=1}^\infty$ of points in $U$ which converges to $a$. By Lemma~\ref{lem:continuity}(\ref{lem:density}), there is a sequence $(x_m)_{m=1}^\infty$ of points in $S'$ such that each $x_m$ corresponds to $a_m$ and $x=\lim_{m\to \infty} x_m$. As each $a_m$ lies in $U$ we have that each $x_m$ lies in $U'$, which establishes density. 

That $U'$ is open in $S'$ follows similarly. Suppose for a contradiction that $U'$ is not open in $S'$. Thus there is some $x=(x_1,\ldots, x_n)\in U'$ corresponding to some  $a=(a_1,\ldots, a_n)\in U$ and a sequence $(x_m)_{m=1}^\infty$ in $S'\setminus U'$ which converges to $x$. By Lemma~\ref{lem:continuity}(\ref{lem:open}), there is a sequence $(a_m)_{m=1}^\infty$ in $S$ such that $x_m$ corresponds to $a_m$ and $\lim_{m\to\infty} a_m=a$. As each $x_m$ lies in $S'\setminus U$ we have that each $a_m$ lies in $S\setminus U$. This contradicts that $U$ is open in $S$, which completes the proof.
\end{proof}

We now prove Lemma~\ref{lem:continuity}.

\begin{proof}[Proof of Lemma~\ref{lem:continuity}]

To prove (a), suppose that $x=(x_1,\ldots, x_n)$ is a sequence of affinely spanning points in $\K^{n-d-1}$ which corresponds to some sequence $a=(a_1,\ldots, a_n)$ of affinely spanning points in $\K^d$. Let $A$ be the $(d+1)\times (n+1)$ matrix in the construction of the Gale dual for $a_1,\ldots, a_n,a_{n+1}=\frac{1}{n}\sum_{i=1}^n a_i$ in $\K^d$, and let $B$ be the $(n-d)\times (n+1)$ matrix whose columns are the conjugates of $g_1=(x_1,1),\ldots, g_n=(x_n,1), g_{n+1}=-\sum_{i=1}^n g_i$. The rows of $B$ form a basis for $\ker A$ by the construction of the Gale transform. For each $a_m=(a_1^m,\ldots, a_n^m)$, let $a_{n+1}^m=\frac{1}{n}\sum_{i=1}^n a_i^m$ and let $A_m$ be the matrix in the construction of the Gale dual of  $a_1^m,\ldots, a_{n+1}^m$.

For each $m$ choose an orthogonal basis $b_1^m,\ldots, b_{n-d}^m$ for $\ker A_m$ such that $b_1^m,\ldots, b_{n-d-1}^m$ are unit vectors and $b_{n-d}^m=(1,\ldots, 1,-n)$, and let $B_m$ be the matrix whose rows are the $b_i^m$. After passing to convergent subsequence if necessary, we may assume that each $b_i^m$ converges to some $b_i'$ with $b_{n-d}'=(1,\ldots, 1,-n)$. Since the $a_m$ converge to $a$ we have that the $A_m$ converge to $A$ in the Frobenius norm, from which it is straightforward to show that $b_1',\ldots, b_{n-d-1}', b'_{n-d}$ are elements of $\ker A$. By the choice of basis of $\ker A_m$ it is likewise readily seen that the $b_i'$ form a basis for $\ker A$. We let $B'$ be the matrix whose rows are the $b_i'$. Thus $\lim_{m\to \infty} B_m=B'$.
 
 Now let $T$ be the invertible $(n-d)\times (n-d)$ matrix such that $TB'=B$. As $b_{n-d}'=b_{n-d}$, the last row of $T$ is the $(n-d)$-th standard basis vector $\mathbf{e}_{n-d}\in \K^{n-d}$. For each $m$ the rows of $TB_m$ are still a basis for $\ker A_m$, and since the last row of $T$ is $\mathbf{e}_{n-d}$ the last row of $TB_m$ is still $(1,\ldots, 1,-n)$. The Gale dual of $a_1^m,\ldots, a_n^m,a_{n+1}^m=\frac{1}{n}\sum_{i=1}^na^m_i$ with respect to this basis for $\ker A_m$ is therefore of the form $g_1^m=(x_1^m,1) ,\ldots, g_n^m=(x_n^m,1), g^m_{n+1}=-\sum_{i=1}^n g_i^m$. Thus each $x_m=(x_1^m,\ldots, x_n^m)$ corresponds to $a_m$. Since the $B_m$ converge to $B'$ we have that the $TB_m$ converge to $B$ and therefore that $\overline{TB_m}$ converges to $\overline{B}$. Thus the $g_i^m$ converge to $g_i$ for each $i\in [n+1]$ and hence the $x_m$ converge to $x$.

 The proof of part (b) is analogous, where now we reverse the steps in the construction of the Gale dual as in the proof of Proposition~\ref{prop:Gale basic}. So suppose that $x=(x_1,\ldots, x_n)$ is a sequence of affinely spanning points in $\K^{n-d-1}$ which corresponds to some sequence  $a=(a_1,\ldots, a_n)$ of affinely spanning points in $\K^d$, and moreover that there is a sequence $(x_m)_{m=1}^\infty$ of affinely spanning points in $\K^d$ which converges to $x$. Let $A$ be the $(d+1)\times (n+1)$ matrix in the construction for the Gale transform of $a_1,\ldots,a_n, a_{n+1}=\frac{1}{n}\sum_{i=1}^n a_i$ and let $B$ be the $(n-d)\times (n+1)$ matrix whose columns are the conjugates of $g_1=(x_1,1),\ldots, g_n=(x_n,1), g_{n+1}=-\sum_{i=1}^n g_i$. We have that $\Row B=\ker A$ by the construction of the Gale transform, so $\ker B=\Row A$ and thus the rows $w_1,\ldots, w_d,w_{d+1}=(1,\ldots, 1)$ of $A$ form a basis for $\ker B$. For each $m$, let $B_m$ be the matrix whose columns are the conjugates of $g_1^m=(x_1^m,1),\ldots, g_n^m=(x_n^m,1),g^m_{n+1}=-\sum_{i=1}^n g^m_i$ and choose an orthogonal basis $w_1^m,\ldots,w_{d+1}^m$ for $\ker B_m$ such that $w^m_{d+1}$ is the all 1 vector and the other $w_i^m$ have norm 1. After passing to a subsequence if necessary we may assume that each $w_i^m$ converges to some $w_i'$, where $w_{d+1}'$ is the all one vector. As the $x_m$ converge to $x$ we have that the $\overline{B_m}$ converge to $\overline{B}$ and so that the $B_m$ converge to $B$. It follows as before that $w_1',\ldots, w_{d+1}'$ is a basis for $\ker B$. 
 
 Now let $A'$ be the matrix whose rows are $w_1',\ldots, w'_{d+1}$, and likewise for each $m$ let $A_m$ be the matrix whose rows are $w_1^m,\ldots, w_{d+1}^m$. We let $T$ be the invertible $(d+1)\times (d+1)$ matrix such that $TA'=A$. Since the $w_i^m$ converge to $w_i'$, the $A_m$ converge to $A'$ and so the $TA_m$ converge to $A$. For each $m$ the rows of $TA_m$  still form a basis for $\ker B_m$, and, since $w'_{d+1}=w_{d+1}$, the last row of $T$ is $\mathbf{e}_{d+1}$ and so the last row of $TA_m$ is still the all one vector. Finally, for each $m$ let $a_1^m,\ldots, a_{n+1}^m$ be the columns of the matrix formed by removing the final row from $TA_m$ and let $a_m=(a_1^m, \ldots, a_n^m)$. The proof of Proposition~\ref{prop:Gale basic} shows that $g_1^m,\ldots, g_{n+1}^m$ is the Gale dual of $a_1^m,\ldots, a_{n+1}^m$ obtained by letting the rows of $B_m$ serve as the chosen basis for $\ker TA_m$. Because the last row of $B_m$ is $(1,\ldots, 1,-n)$ we have that $u_{n+1}=\frac{1}{n}\sum_{i=1}^n u_i$ for each $u=(u_1,\ldots, u_{n+1})\in \ker B_m$ and therefore that $a_{n+1}^m=\frac{1}{n}\sum_{i=1}^n a^m_i$. Thus each $x_m$ corresponds to $a_m$. Since the $TA_m$ converge to $A$ we have that the $a_m$ converge to $a$.  
\end{proof}

 Next, we show that for a typical point set the open half-flats of the $r$-fans of Theorems~\ref{thm:Dolnikov} and ~\ref{thm:colorful} contain a sizable number of points. 

\begin{proposition}
\label{prop:generic distribution} Let $d,m\geq 1$ be integers and let $\K=\R$ or $\C$. Let $n\geq (r-1)(d+m+1)+1$ if $\K=\R$, let $n\geq (r-1)(2d+m+1)+1$ if $\K=\C$, and suppose that $X$ is a typical set of $n$ points in $\K^{n-d-1}$. Then the $r$-fans $F_r=\cup_{j=1}^r B_j$ guaranteed by Theorem ~\ref{thm:Dolnikov} and~\ref{thm:colorful} satisfy the property that $\sqcup_{j=1}^r \Int(B_j)$ contains at least $(r-1)(d+1)+1$ points of $X$ when $\K=\R$, and at least $(r-1)(2d+1)+1$ points of $X$ when $\K=\C$. 
\end{proposition}

\begin{proof}[Proof of Proposition~\ref{prop:generic distribution}]
Suppose that $X=\{x_1,\ldots, x_n\}$ is a typical point set in $\K^{n-d-1}$. Thus $x_1,\ldots, x_n$ corresponds to a sequence of points $a_1,\ldots, a_n$ in $\K^d$ in strong general position. The distributing $r$-fan $F_r=\cup_{j=1}^r$ guaranteed by the proofs of Theorem~\ref{thm:Dolnikov} and Theorem~\ref{thm:colorful} was given by a Tverberg $r$-tuple $(\sigma_1,\ldots, \sigma_r)$ for the linear map $f\colon \Delta_{n-1}\rightarrow \K^d$ determined by the $a_i$. Moreover, we had $|\Int(B_j)\cap X|=|\Ver(\sigma_j)|$ for all $j\in [r]$. As the $a_i$ are in strong general position, we have $\sum_{j=1}^r |\Int(B_j)\cap X|=\sum_{j=1}^r |\Ver(\sigma_j)|\geq (r-1)(d+1)+1$ when $\K=\R$ and $\sum_{j=1}^r |\Int(B_j)\cap X|\geq (r-1)(2d+1)+1$ when $\K=\C$.
\end{proof}

\subsection{Asymptotics of Theorem~\ref{thm:Dolnikov}}
\label{sec:optimal}

Finally, we prove the near optimality of the real case of Theorem~\ref{thm:Dolnikov} with respect to dimension for all prime powers $r\geq 3$ when $m\geq 2$, and likewise in the complex setting when $r=2$ and $m\geq 1$. For this, it suffices to consider the near optimality of Theorems~\ref{thm:real equidistribute} and ~\ref{thm:complex equidistribute}, respectively.

\begin{proposition}
\label{prop:optimal} Let $d\geq 1$ and $m\geq 2$ be integers. 

\begin{compactenum}[(a)]
\item Let $r\geq 3$, let $k\geq 0$, and let $n=(r-1)(d+m+1)+k+1$. If $\ell>2+\frac{k+2}{r-2}$, then there exists a set $X$ of $n+\ell$ affinely spanning points in $\R^{n-d-1}$ and an $m$-coloring of $X$ which cannot be equidistributed by any conical $r$-fan.

\item Let $k\geq 0$ and let $n=2(d+1)+m+k$. If $\ell>k+3$ then there exists a set $X$ of $n+\ell$ complex affinely spanning points in $\C^{n-d-1}$ and an $m$-coloring of $X$ which cannot be equidistributed by any complex regular $2$-fan in $\C^{n-d-1}$.
\end{compactenum}
\end{proposition}

\begin{proof}[Proof of Proposition~\ref{prop:optimal}] 

In the real case, let $n=(r-1)(d+m+1)+k+1$ where $k\geq 0$ and assume that $\ell>2+\frac{k+2}{r-2}$. Suppose that $a_1,\ldots, a_n, a_{n+\ell-1}$ are points in strong general position in $\R^{d+\ell-1}$ and let $x_1,\ldots, x_{n+\ell-1}$ be their Gale dual in $\R^{n-d-1}$. Since the $x_i$ linearly span $\R^{n-d-1}$ and sum to zero they also affinely span $\R^{n-d-1}$. Consider the sequence $a_1'=(a_,1),\ldots, a_{n+\ell-1}'=(a_{n+\ell-1},1)$ in $\R^{d+\ell}$ and append to this sequence any point $a_{n+\ell}'$ in $\R^{d+\ell}$ such that $a_1',\ldots, a'_{n+\ell}$ affinely spans $\R^{d+\ell}$. We now claim that, for a suitable choice of basis, the Gale transform of the $a_i'$ is $x_1,\ldots, x_{n+\ell-1}, x_{n+\ell}$, where $x_{n+\ell}=0$. To see this, let $A$ be the $(d+\ell)\times (n+\ell-1)$ matrix in the construction of the Gale transform of $a_1,\ldots, a_{n+\ell-1}$, and let $b_1,\ldots, b_{n-d-1}\in\R^{n+\ell-1}$ be a choice of basis for $\ker A$ so that $x_1,\ldots, x_{n+\ell-1}$ are the columns of the matrix $B$ whose rows are the $b_i$. Letting $A'$ be the $(d+\ell+1)\times (n+\ell)$ matrix in the construction of the Gale transform of the $a_i'\in \R^{d+\ell}$, we may choose $b_1'=(b_1,0),\ldots, b_{n-d-1}'=(b_{n-d-1},0)$ as a basis for $\ker A'$. Thus the Gale transform of $a_1', \ldots, a'_{n+\ell}$ in $\R^{n-d-1}$ is $x_1,\ldots, x_{n+\ell-1},x_{n+\ell}=0$. 

We now show that there exists an $m$-coloring of $X=\{x_1,\ldots, x_{n+\ell}\}$ which cannot be equidistributed by any conical $r$-fan in $\R^{n-d-1}$. Namely, let $X=X_1\sqcup\cdots\sqcup X_m$ be an $m$-coloring such that $|X_i|=r-1$ for all $i\geq 2$ and $x_{n+\ell}=0$ is in $X_m$. Now suppose that $F_r=\cup_{j=1}^r B_j$ is a conical $r$-fan which equidistributes this coloring. As $|\Int(B_j)\cap X_i|\leq |X_i|/r$ for all $i\in [m]$ and $j\in [r]$, all the points of $X\setminus X_1$ must lie in the center $C$ of $F_r$. In particular, $F_r$ is linear. Now let $f\colon\Delta_{n+\ell-1}\rightarrow \R^{d+\ell}$ be the linear map determined by the $a_i'$ and let $[n+\ell]=C_1\sqcup\cdots\sqcup C_m$ be the $m$-coloring of the vertex set of $\Delta_{n+\ell-1}$ corresponding to the $m$-coloring of $X$. By Lemma~\ref{lem:linear conical fans} there is a proper Tverberg $r$-tuple $(\sigma_1,\ldots, \sigma_r)$ for $f$ such that each $I_j:=\Ver(\sigma_j)$ corresponds to $X\cap \Int(B_j)$ and $[n+\ell-1]\setminus \sqcup_{j=1}^r I_j$ corresponds to $C\cap X$. In particular, $|C_i\cap \Ver(\sigma_j)|=|X_i\cap \Int(B_j)|$ for all $j\in [r]$ and all $i\in [m]$. Thus each $I_j$ is contained in the color class $C_1$, which consists of $n+\ell-(m-1)(r-1)=(r-1)(d+2)+k+\ell+1$ points. Letting $\Delta_{C_1}$ be the simplex generated by $C_1$, we have that the image of the restriction $f\colon\Delta_{C_1}\rightarrow \R^{d+\ell}$ of $f$ is contained in $\R^{d+\ell-1}\times \{1\}$. As the points $a_1'=(a_1,1),\ldots, a_{n+\ell-1}'=(a_{n+\ell-1},1)$ are in strong general position in $\R^{d+\ell-1}\times \{1\}$ and $k<(r-1)(\ell-2)-\ell$, we have 
\begin{align*} 
\codim(\cap_{j=1}^r \Aff(f(\sigma_j))) &  =\sum_{j=1}^r \codim(\Aff(f(\sigma_j))\\ &=r(d+\ell)-\sum_{j=1}^r |I_j|
 \geq r(d+\ell)-[(r-1)(d+2)+k+\ell+1]>d+\ell-1.
\end{align*}
Thus $\cap_{j=1}^r \Aff(f(\sigma_j))=\emptyset$, contradicting that $(\sigma_1,\ldots, \sigma_r)$ is a Tverberg  $r$-tuple for $f$. Therefore there is no conical $r$-fan in $\R^{n-d-1}$ which equidistributes the $m$-coloring of $X$. 

The proof in the complex case when $r=2$ is analogous. Let $n=2(d+1)+m+k$ where $k\geq0$. As before, there exist $n+\ell$ complex affinely spanning points $a_1,\ldots, a_{n+\ell}$ in $\C^{d+\ell}$ whose Gale dual $x_1,\ldots, x_{n+\ell}$ in $\C^{n-d-1}$ satisfies $x_{n +\ell}=0$ and for which $a_1,\ldots, a_{n+\ell-1}$ are in strong general position in $\C^{d+\ell-1}\times \{1\}$. In particular, $X$ complex affinely spans $\C^{n-d-1}$. As before, we choose an $m$-coloring of $X$ so that $|X_i|=1$ for all $2\leq i\leq m$ and $0\in X_m$. If $F_2$ were a complex regular $2$-fan which equidistributes the $m$-coloring, then $F_2$ would be linear and so by Remark~\ref{rem:equivalence} the linear map $f\colon \Delta_{n+\ell-1}\rightarrow \C^{d+\ell}$ determined by the $a_i$ would have a Tverberg tuple $(\sigma_1,\sigma_2)$ where each $I_j:=\Ver(\sigma_j)$ is contained in $C_1$. Thus $f|_{\Delta_{C_1}}\colon \Delta_{C_1}\rightarrow \C^{d+\ell}$  maps to $\C^{d+\ell-1}\times \{1\}$. As $a_1,\ldots, a_{n+\ell-1}$ are in strong general position in $\C^{d+\ell-1}\times \{1\}$, the condition that $k<\ell-3$ gives $\codim(\Aff(f(\sigma_1\cap\sigma_2))> 2(d+\ell-1)$. We conclude as before that no equidistributing complex regular $2$-fan exists. 
\end{proof}

Corollary~\ref{cor:values} is now an immediate consequence of Theorem~\ref{thm:real equidistribute} and Proposition~\ref{prop:optimal}.

\begin{proof}[Proof of Corollary~\ref{cor:values}] Let $c=(r-1)(m+1)$ and let $d=(r-2)s+t+c$, where $m\geq 2$, $s\geq 1$, and $t\in\{0,\ldots, r-3\}$. Letting $n=d+s+1$ and applying Theorem~\ref{thm:real equidistribute}, we have that any $m$-coloring of any affinely spanning set of $n=(r-1)(s+m+1)+t+1$ points in $\R^{n-s-1}=\R^d$ can be equidistributed by a conical $r$-fan in $\R^d$. Thus $n(r,m,d)\geq n$. On the other hand, Proposition~\ref{prop:optimal} shows that $n(r,m,d)\leq n+\ell-1$ so long as $\ell>2+\frac{t+2}{r-2}$. But $0\leq t\leq r-3$, so $n(3,m,d)\leq n+4$ and $n(r, m, d)\leq n+3$ when $r>3$.
\end{proof}

\begin{remark} 
\label{rem:complex}
As in the real case, we denote by $n_\C(r,m,d)$ the maximum number $n$ such that any $m$-coloring of any set of $n$ complex affinely spanning points in $\C^d$ can be equidistributed by a complex regular $r$-fan in $\C^d$. Letting $d=s+m+1$ with $s\geq 1$ and setting $n=d+s+1$, Theorem~\ref{thm:complex equidistribute} together with Proposition~\ref{prop:optimal} give the bounds $n\leq n_\C(2,m,d)\leq n+3$ for complex regular $2$-fans when $m\geq 2$. When $m=1$, Theorem~\ref{thm:complex equidistribute} guarantees that any collection of $n=2s+3$  affinely spanning points in $\C^d\cong \R^{2d}$ can be equidistributed by a complex regular $2$-fan. On the other hand, no more than $2d=2s+4$ points in general position in $\R^{2d}$ can lie on an affine hyperplane. Thus $n\leq n_\C(2,1,d)\leq n+1$.  \end{remark}

\section*{Acknowledgements} 

The authors acknowledge the support provided by the 2024 Bard Science Research Institute. The authors are especially grateful to the anonymous referees for their thoughtful comments and suggestions which greatly improved the exposition of the manuscript. They also thank Florian Frick for helpful conversations.

\end{document}